\documentclass{article}
\usepackage[latin1]{inputenc}
\usepackage{amsmath}
\usepackage{amsthm,amssymb,amsfonts}
\usepackage{graphicx}
\usepackage{verbatim}
\usepackage{color,cite}
\usepackage{latexsym}
\usepackage{lscape}
\usepackage[T1]{fontenc}
\usepackage[all,cmtip]{xy}
\usepackage{tikz}
\usepackage{xcolor}
\usetikzlibrary{patterns}

\newcommand{\nc}{\newcommand}

\nc\cA{\mathcal{A}}
\nc\cB{\mathcal{B}}
\nc\cC{\mathcal{C}}
\newcommand{\cE}{\mathcal{E}}

\newcommand{\cH}{\mathcal{H}}

\newcommand{\cM}{\mathcal{M}}

\nc\bK{\mathbb{K}}
\nc\bL{\mathbb{L}}
\nc\bM{\mathbb{M}}
\nc\bN{\mathbb{N}}
\nc\bO{\mathbb{O}}
\nc\bP{\mathbb{P}}
\nc\bQ{\mathbb{Q}}
\nc\bR{\mathbb{R}}
\nc\bS{\mathbb{S}}
\nc\bT{\mathbb{T}}
\nc\bU{\mathbb{U}}
\nc\bV{\mathbb{V}}
\nc\bW{\mathbb{W}}
\nc\bZ{\mathbb{Z}}

\theoremstyle{plain}
\newtheorem{theorem}{Theorem}[section]
\newtheorem{lemma}[theorem]{Lemma}
\newtheorem{proposition}[theorem]{Proposition}
\newtheorem{corollary}[theorem]{Corollary}

\newtheorem*{definition}{Definition}
\newtheorem{example}[theorem]{Example}

\newtheorem*{question}{Question}

\newtheorem*{remark}{Remark}

\newtheorem*{convention}{Convention}

\title{Affine Manifolds and Zero Lyapunov Exponents in Genus $3$}
\author{David Aulicino\thanks{This material is based upon work supported by the National Science Foundation under Award No. DMS - 1204414, and later by the ERC Starting Grant ``Quasiperiodic'' of Artur Avila.}}
\date{}

\begin{document}

\newcommand{\splin}{$\text{SL}_2(\mathbb{R})$}
\newcommand{\spolin}{$\text{SO}_2(\mathbb{R})$}

\maketitle

\begin{abstract}
In previous work, the author fully classified orbit closures in genus three with maximally many (four) zero Lyapunov exponents of the Kontsevich-Zorich cocycle.  In this paper, we prove that there are no higher dimensional orbit closures in genus three with any zero Lyapunov exponents.  Furthermore, if a Teichm\"uller curve in genus three has two zero Lyapunov exponents in the Kontsevich-Zorich cocycle, then it lies in the principal stratum and has at most quadratic trace field.  Moreover, there can be at most finitely many such Teichm\"uller curves.
\end{abstract}

\tableofcontents

\section{Introduction}

It is well known that the Teichm\"uller geodesic flow in the moduli space of Riemann surfaces admits a collection of $6g-6$ Lyapunov exponents.  Since \cite{KZHodge}, it was realized that the ``non-trivial exponents'' could be understood via the $2g$ Lyapunov exponents of a cocycle over the Teichm\"uller geodesic flow on the first cohomology bundle over the moduli space, now known as the Kontsevich-Zorich cocycle.   This perspective led to the work of \cite{ForniDev} and \cite{AvilaVianaSimp}, who proved the spectral gap and simplicity of the spectrum of Lyapunov exponents with respect to the Masur-Veech measures on strata of holomorphic Abelian differentials.  Explicit formulas for sums of the exponents were found in \cite{EskinKontsevichZorich2}.  However, except for certain special situations \cite{EskinKontsevichZorichSqTiled, BouwMoeller, WrightSchwarzAbSqTilSurfs, WrightSchwarzVWBMCurves}, formulas for individual exponents are not known.

Since then, examples were found of Teichm\"uller curves with Lyapunov exponents equal to zero \cite{ForniHand, ForniMatheusZorichLyapSpectHodge, ForniMatheusZorichZeroLyapExpsHodge, EskinKontsevichZorichSqTiled, WrightSchwarzAbSqTilSurfs}, thus demonstrating that the positivity of the top $g$ Lyapunov exponents of the cocycle is false in general.  It is possible to construct higher dimensional affine manifolds with zero Lyapunov exponents as well, e.g. \cite{GrivauxHubertFullDegQuad}.

The goal of this paper is to try to get a complete understanding of \splin -orbit closures in the moduli space of genus three Abelian differentials with zero Lyapunov exponents.  Since the Kontsevich-Zorich cocycle respects the symplectic intersection form on cohomology, the spectrum of $2g$ exponents is symmetric.  Furthermore, the largest exponent is always equal to $1$.  Therefore, in genus three, the number of Lyapunov exponents equal to zero is either $0$, $2$, or $4$.

The case of four zero Lyapunov exponents is completely understood through the work of \cite{AulicinoCompDegKZ, AulicinoCompDegKZAIS, MollerShimuraTeich}.  One of the key ingredients in all of these works is that the existence of maximally many zero Lyapunov exponents implies that a specific mechanism produced them, namely a maximal Forni subspace in the cohomology bundle.

A conjecture explaining the mechanisms for producing zero exponents was proposed in \cite{ForniMatheusZorichZeroLyapExpsHodge}.  Recently, \cite{FilipZeroExps} proved that if the cocycle admits a zero Lyapunov exponent, then it must satisfy one of several properties on a list.  In genus three, this list coincides exactly with the mechanisms proposed in \cite{ForniMatheusZorichZeroLyapExpsHodge}.  Precisely speaking, in genus three the only mechanisms for producing zero Lyapunov exponents of the Kontsevich-Zorich cocycle are the $H^1$ bundle having non-trivial Forni subspace, i.e. the algebraic hull has a compact factor, or the algebraic hull of the Kontsevich-Zorich cocycle having an $\text{SU}(p,q)$ factor, where $p > q$.  This follows simply from the fact that the other mechanisms require the dimension of $H^1$ to be greater than six.  In fact, this can be strengthened further.

\begin{proposition}
\label{Gen3ZeroImpForni}
Let $\mathcal{M}$ be an \splin -orbit closure in the moduli space of Abelian differentials on genus three surfaces.  If the Kontsevich-Zorich spectrum of $\mathcal{M}$ has Lyapunov exponents equal to zero, then the cohomology bundle over $\mathcal{M}$ admits a Forni subspace of dimension equal to the number of zero Lyapunov exponents of the Kontsevich-Zorich cocycle.
\end{proposition}

The following theorem mostly confirms a conjecture of Carlos Matheus that the existence of a single zero Lyapunov exponent in genus three implies that there are maximally many.

\begin{theorem}
\label{MainThmTotal}
Let $\mathcal{M}$ be an \splin -orbit closure in the moduli space of Abelian differentials on genus three surfaces.  If the Kontsevich-Zorich spectrum of $\mathcal{M}$ has exactly two zero Lyapunov exponents, then $\mathcal{M}$ is a Teichm\"uller curve in the principal stratum $\mathcal{H}(1,1,1,1)$ with a $2$-dimensional Forni subspace, and trace field of degree at most two over $\mathbb{Q}$.  Furthermore, there are at most finitely many such Teichm\"uller curves.
\end{theorem}

In fact, the Eierlegende Wollmilchsau is the only known Teichm\"uller curve in $\mathcal{H}(1,1,1,1)$ with zero Lyapunov exponents.  By \cite{McMullenDecagonUnique}, the decagon generates the only Teichm\"uller curve in $\mathcal{H}(1,1)$ with quadratic trace field.  By inspection, the core curves of the cylinders of any unramified double cover of the decagon span a Lagrangian subspace of homology, thus they have no zero exponents by the Forni Criterion \cite{ForniCriterion}.  Furthermore, computer experiments by Vincent Delecroix have shown that there are no square-tiled surfaces with fewer than $19$ squares with zero exponents, other than the Eierlegende Wollmilchsau.

\begin{question}
Are there any Teichm\"uller curves in the principal stratum in genus three with exactly two zero Lyapunov exponents?
\end{question}

The proof of this theorem requires combining nearly all known techniques for studying zero Lyapunov exponents of the Kontsevich-Zorich cocycle.  However, new techniques must be developed because all of the present methods do not suffice to prove this theorem.  The most significant new technique is the conversion of the fact that the Forni bundle is constant into information about the flat surfaces contained in the orbit closure.

The starting point is the application of \cite{FilipZeroExps} in Proposition \ref{Gen3ZeroImpForni} to prove that zero Lyapunov exponents imply the existence of a Forni subspace.  From here the proof is split into two cases: rank one and rank two affine manifolds.  The work of \cite{YuZuoWeierstrassFiltLyapExps} followed by \cite{BainbridgeHabeggerMollerHNFilt} on the Harder-Narasimhan filtration, allows us to exclude Teichm\"uller curves outside of the principal stratum cf. Theorem \ref{NoZeroInNonVary}.  However, the Harder-Narasimhan filtration perspective says nothing about Lyapunov exponents of rank one affine manifolds in the principal stratum.

Next, we consider the flat structure of translation surfaces in such an affine manifold, and the Forni Criterion allows us to enumerate the configurations of cylinders, cf. Table \ref{MainConfigTable}.  However, there are well over $30$ cylinder diagrams so the combinatorics of studying each individual diagram is impossible, especially given the number of potentially free parameters on each cylinder diagram.  Thus, the Forni Criterion alone cannot prove Theorem \ref{MainThmTotal}.

For rank two affine manifolds, a simple argument excludes them outside of the principle stratum, cf. Theorem \ref{NoAISLowerStrata}, and the cylinder deformation theorem of \cite{WrightCylDef} combined with the degeneration arguments of \cite{AulicinoCompDegKZ, AulicinoCompDegKZAIS}, suffice to exclude them from the principal stratum, cf. Section \ref{Rank2PrinStrSect}.  Alternatively, recent work of \cite{NguyenWright, AulicinoNguyenWright, AulicinoNguyenGen3TwoZeros} have classified rank two manifolds in strata with at most two zeros.  It also follows from this classification and the Forni Geometric Criterion that there are no rank two manifolds in strata with at most two zeros with any zero Lyapunov exponents.  Thus, it suffices to focus on higher dimensional rank one affine manifolds.

The first novel idea in this paper is to use the cylinder deformations of \cite{WrightCylDef} in the case of rank one affine manifolds to make claims about Lyapunov exponents.  The results of \cite{WrightCylDef} allow us to describe deformations of translation surfaces in a rank one affine manifold that is not a Teichm\"uller curve, cf. Section \ref{RELDefSect}.  The results of Section \ref{RELDefSect} were not explicitly stated in \cite{WrightCylDef}, which is why we record them here.  We note that they apply to genus $g$ translation surfaces, and no assumptions are made about Lyapunov exponents.  In fact, the author hopes that these results may have applications in the study of rank one affine manifolds outside of questions about Lyapunov exponents.  For example, they may prove useful for questions of weak mixing for the straight line flow on translation surfaces in rank one affine manifolds, and potentially for classifying higher dimensional non-arithmetic rank one manifolds in low genus following the theme of \cite{NguyenWright, AulicinoNguyenWright, AulicinoNguyenGen3TwoZeros}.

In this paper, we extract information about translation surfaces from the theorem of \cite{AvilaEskinMollerForniBundle} that the Forni subspace is a constant bundle.  This follows a recent theme that has emerged following the work of \cite{EskinMirzakhaniInvariantMeas, EskinMirzakhaniMohammadiOrbitClosures}: transform given properties of an affine manifold into properties about the translation surfaces contained within it.  The greatest breakthrough along these lines was the work of \cite{WrightCylDef} that converts the rank of an affine manifold into statements about deforming cylinders on translation surfaces within that manifold.  This paper converts the assumption that the Forni subspace is constant into a statement that allows us to compare the homology classes of cylinders that are \emph{not} parallel, and uses this to restrict the dimension of the Forni subspace.  The Forni Criterion only concerns parallel cylinders, and this will not suffice for our purposes when the affine manifold has small dimension.  In other words, when the affine manifold has small dimension the deformations are so restricted that it is not at all clear if there exists a translation surface where the Forni Criterion alone yields the desired result.

The finiteness result follows from the equidistribution results of affine manifolds due to \cite{EskinMirzakhaniMohammadiOrbitClosures}.  All of these results use \cite{EskinMirzakhaniInvariantMeas, EskinMirzakhaniMohammadiOrbitClosures} in an essential way so that we know that the orbit closure is an affine manifold.

\

\noindent \textbf{Acknowledgments.}  The author is very grateful to Alex Eskin for suggesting the problem, contributing ideas, and having numerous discussions throughout the development of the work.  The author is also very grateful to Martin M\"oller for pointing out that the results on the Harder-Narasimhan filtration could be applied to this problem.  He would also like to thank Martin M\"oller along with his co-authors Matt Bainbridge and Philipp Habegger, for including Theorem \ref{NoZeroInNonVary} in their paper.  He is grateful to Simion Filip for sharing his work on zero exponents, and for helpful discussions.  He would like to thank Alex Wright for clarifying the ideas in Section 3.  The author is grateful to Matt Bainbridge, Barak Weiss, and Giovanni Forni for many helpful discussions.

\section{Outline of Results}

In this section we recall all of the basic definitions and results needed for this paper.  Then we give the complete proof of the main theorem with the exception of two technical theorems, Theorems \ref{NoRank2PrinStratum} and \ref{Rank1PrinStratum}, whose proofs consist of the bulk of this paper.

\subsection{Preliminaries}

\noindent \textbf{Strata of Abelian Differentials:}  Let $M$ be a translation surface of genus $g \geq 2$.  Equivalently, let $M = (X, \omega)$, where $X$ is a Riemann surface of genus $g$ carrying an Abelian differential $\omega$.  If $\omega$ is holomorphic, then the total order of its zeros is $2g-2$.  Let $\kappa$ denote a partition of $2g-2$.  Consider the set of all pairs $(X,\omega)$ up to equivalence under the mapping class group, where the orders of the zeros of $\omega$ are prescribed by $\kappa$.  This set, denoted $\mathcal{H}(\kappa)$, is called a \emph{stratum of Abelian differentials}.  Let $\Omega\cM_g$ denote the union of all strata of Abelian differentials on surfaces of genus $g$, i.e. $\Omega\cM_g$ is exactly the moduli space of Riemann surfaces of genus $g$ with the bundle of Abelian differentials.

\

\noindent \textbf{\splin ~Action:} The translation surface $M$ can be regarded as a collection of polygons in the plane, where every side has a corresponding parallel side to which it is identified.  Since \splin ~acts on the plane, there is a natural action by \splin ~on $\mathcal{H}(\kappa)$.  Furthermore, the action preserves the area of $M$.  The action by the subgroup of diagonal matrices is called the \emph{Teichm\"uller geodesic flow}.

\

\noindent \textbf{Period Coordinates:} The stratum $\mathcal{H}(\kappa)$ admits natural local charts given by the period coordinate mapping to $\mathbb{C}^n$.  Let $(X,\omega) \in \mathcal{H}(\kappa)$.  Let $\Sigma$ denote the set of zeros of $\omega$.  If we fix a basis $\{\gamma_1, \ldots, \gamma_n\}$ for $H_1(X,\Sigma, \mathbb{Z})$, then we get a local map into relative cohomology
$$\Phi: \begin{array}{ccc}\mathcal{H}(\kappa) & \rightarrow & H^1(X,\Sigma, \mathbb{C}) \cong \mathbb{C}^n \\ (X,\omega) & \mapsto & \left( \int_{\gamma_1} \omega, \ldots, \int_{\gamma_n} \omega \right) \end{array}$$

\

\noindent \textbf{Affine Manifolds:} Period coordinates give local linear coordinates on the stratum.  An \emph{affine \splin -invariant manifold} $\mathcal{M} \subset \mathcal{H}(\kappa)$ is an immersed manifold that is locally linear in period coordinates and \splin -invariant.  It was proven in \cite{EskinMirzakhaniInvariantMeas} and \cite{EskinMirzakhaniMohammadiOrbitClosures}, that the closure of every \splin ~orbit in $\mathcal{H}(\kappa)$ is an affine \splin -invariant submanifold $\mathcal{M}$ that admits a finite \splin -invariant measure $\nu$ after restricting to the unit area surfaces in $\mathcal{M}$, and $\nu$ is affine with respect to period coordinates.\footnote{Actually, $\mathcal{M}$ is an immersed orbifold, but after passing to a finite cover, it is an immersed manifold.}  We will often abbreviate terminology and call $\mathcal{M}$ an affine manifold.

The tangent space of $\mathcal{M}$ can be given in period coordinates as a subspace $T_{\mathbb{C}}(\mathcal{M}) \subset H^1(X,\Sigma,\mathbb{C})$, where the inclusion into first cohomology is seen by considering the period coordinates as the derivative map.  The tangent space satisfies $T_{\mathbb{C}}(\mathcal{M}) = \mathbb{C} \otimes T_{\mathbb{R}}(\mathcal{M})$, where $T_{\mathbb{R}}(\mathcal{M}) \subset H^1(X,\Sigma,\mathbb{R})$.

\

\noindent \textbf{Teichm\"uller Curve:} A \emph{Veech surface} is a translation surface with the property that the derivatives of its affine diffeomorphisms form a lattice subgroup of \splin .  The orbit of a Veech surface under \splin ~is closed, and it is called a \emph{Teichm\"uller curve}.

\

\noindent \textbf{Rank of an Affine Manifold:}   Let $p: H^1(X,\Sigma,\mathbb{R}) \rightarrow H^1(X,\mathbb{R})$ be the natural projection to absolute cohomology.  By \cite{AvilaEskinMollerForniBundle}, the projection of the tangent space of an affine manifold to absolute cohomology is symplectic, whence even dimensional, cf. Theorem \ref{ForniConst}.  In \cite{WrightCylDef}, the (cylinder) rank of an affine invariant manifold $\mathcal{M}$ is defined to be
$$\text{Rank}(\mathcal{M}) = \frac{1}{2}\dim_{\mathbb{R}} p(T_{\mathbb{R}}(\mathcal{M})).$$

\

\noindent \textbf{Field of Affine Definition:} Introduced in \cite{WrightFieldofDef}, the \emph{field of affine definition} $\textbf{k}(\mathcal{M})$ of an affine manifold $\mathcal{M}$ is the smallest subfield of $\mathbb{R}$ such that $\mathcal{M}$ can be defined in local period coordinates by linear equations in $\textbf{k}(\mathcal{M})$.  It was proven that this is well-defined for every affine manifold and has degree at most $g$ over $\mathbb{Q}$, \cite[Thm. 1.1]{WrightFieldofDef}.  In particular, an affine manifold $\mathcal{M}$ is called \emph{arithmetic} when $\textbf{k}(\mathcal{M}) = \mathbb{Q}$.

\

\noindent \textbf{Lyapunov Exponents:} The bundle $H^1_{\mathbb{F}}$ over $\mathcal{H}(\kappa)$ is the bundle with fibers $H^1(X,\mathbb{F})$ and a flat connection (the Gauss-Manin connection) given by identifying nearby lattices $H^1(X,\mathbb{Z})$ and $H^1(X',\mathbb{Z})$.  If $\mathcal{M}$ is an affine manifold, then the Teichm\"uller geodesic flow acts on $\mathcal{M}$ and thus, induces a flow on $H^1_{\mathbb{R}}$.  This flow is known as the \emph{Kontsevich-Zorich cocycle} (KZ-cocycle).

If we consider orbits under the Teichm\"uller geodesic flow that return infinitely many times to a neighborhood of the starting point, then it is possible to compute the monodromy matrix $A(t)$ at each return time $t$.  By computing the logarithms of the eigenvalues of $A(t)A^T(t)$, normalizing them by twice the length of the geodesic at time $t$, and letting $t$ tend to infinity, we get a collection of $2g$ numbers known as the \emph{spectrum of Lyapunov exponents of the KZ-cocycle}, or \emph{KZ-spectrum} for short.  By the Oseledec multiplicative ergodic theorem, these numbers will not depend on the initial starting point for $\nu$-almost every choice of initial data.  Since cohomology admits a symplectic basis that is respected by the monodromy matrix, the KZ-spectrum is symmetric, so the $2g$ Lyapunov exponents of the KZ-cocycle are
$$1 = \lambda_1^{\nu} \geq \cdots \geq \lambda_g^{\nu} \geq -\lambda_g^{\nu} \geq \cdots \geq -\lambda_1^{\nu} = -1.$$
We will suppress the measure from now on and always assume it to be the canonical measure guaranteed by \cite{EskinMirzakhaniInvariantMeas}.  These exponents exist and are defined for every translation surface in almost every direction by \cite{ChaikaEskinLyapExps}.

\

\noindent \textbf{Forni Subspace:} The \emph{Forni subspace} $F(x) \subset H^1(X,\mathbb{R})$ was formally defined in \cite{AvilaEskinMollerForniBundle}.  The subspace $F(x)$ is the maximal \splin -invariant subspace on which the KZ-cocycle acts by isometries with respect to the Hodge inner product.

\begin{theorem}[\cite{AvilaEskinMollerForniBundle} Thm. 1.3, Thm. 2.4]
\label{ForniConst}
For $\nu$-a.a. $x$, the Forni subspace $F(x)$ of the absolute cohomology subspace is symplectic, constant along an orbit closure, and orthogonal to the projection of the tangent space of the affine manifold into absolute cohomology, with respect to both the Hodge and symplectic inner product.

Furthermore, $p(T_{\mathbb{R}}(\mathcal{M}))(x)$ is symplectic.
\end{theorem}

\noindent \textbf{Determinant Locus:}  Let $\{\theta_1, \ldots, \theta_g\}$ be a basis of Abelian differentials on $X$.  Define the $ij$-component of the \emph{derivative of the period matrix} at $X$ in direction of the tangent vector $\mu_{\omega}$ by the Ahlfors-Rauch variational formula
$$\left(\frac{d\Pi}{\mu_{\omega}}\right)_{ij} = \int_X \theta_i\theta_j\frac{\overline{\omega}}{\omega}.$$
Here $\mu_{\omega}$ is the Beltrami differential dual to $\omega$ and given by the formula $\mu_{\omega} = \overline{\omega}/\omega$.  The \emph{determinant locus}, introduced in \cite{ForniDev}, is the set
$$\mathcal{D}_g = \{(X,\omega) | \text{det}(d\Pi/\mu_{\omega}) = 0\}.$$

We record the following trivial lemma that follows from the definition of infimum.

\begin{lemma}
\label{ForniDimSubMan}
Let $\mathcal{M}'$ and $\mathcal{M}$ be affine manifolds.  If $\mathcal{M}' \subset \mathcal{M}$, then
$$\inf_{x' \in \mathcal{M}'} \dim F(x') \geq \inf_{x \in \mathcal{M}} \dim F(x).$$
\end{lemma}

\subsection{Genus $3$ Surfaces with Non-trivial Forni Subspace}

In this section we prove the main result, cf. Theorem \ref{MainThmForniSp}, for orbit closures with non-trivial Forni subspace.  With the exception of Theorems \ref{NoRank2PrinStratum} and \ref{Rank1PrinStratum}, the proof of Theorem \ref{MainThmForniSp} is entirely contained in this section.  We begin by proving Proposition \ref{Gen3ZeroImpForni}, which allows us to restrict our attention to the Forni subspace.

\begin{proof}[Proof of Proposition \ref{Gen3ZeroImpForni}]
Since the spectrum of Lyapunov exponents is symmetric and the top Lyapunov exponent is equal to $1$, there are either two or four Lyapunov exponents equal to zero.  This proposition was proven in \cite{AulicinoCompDegKZ, AulicinoCompDegKZAIS} for the case of four zero Lyapunov exponents, so it suffices to assume that there are exactly two zero Lyapunov exponents.  By \cite[Cor. 1.3(i)]{FilipZeroExps}, the Lyapunov exponents in the tangent bundle projected to absolute cohomology are all non-zero.  This implies that the rank of the affine manifold must be either one or two.


If the rank of the affine manifold is two, then all of the zero exponents must lie in the complement of the absolute tangent bundle, which must be a Forni subspace by the Forni-Kontsevich formula for the sum of the Lyapunov exponents \cite{ForniDev}.

Finally, if the rank of the affine manifold is one, then \cite{FilipZeroExps} implies that the zero exponents come from a Forni subspace, or from the algebraic hull having an $\text{SU}(p,q)$ factor.  If the algebraic hull has an $\text{SU}(p,q)$ factor, then after complexifying, the algebraic hull must also have an $\text{SU}(q,p)$ factor.  However, the absolute tangent space is a $2$-dimensional flat bundle in the absolute cohomology bundle.  Therefore, an $\text{SL}(2)$ factor splits off of the algebraic hull.  This implies that $p+q = 2$, which obviously means it is impossible to have $p-q=1$ as required to produce a zero Lyapunov exponent without the presence of a Forni subspace.
\end{proof}

By \cite{AvilaEskinMollerForniBundle}, the Forni subspace is orthogonal to the bundle $p(T(\mathcal{M}))$.  This implies that $\mathcal{M}$ must be an affine manifold of rank one or two.  The first theorem for addressing these cases is the result \cite[Prop. 4.5]{BainbridgeHabeggerMollerHNFilt}, which applies to Teichm\"uller curves.  We restate their theorem here in the language of this paper.

\begin{theorem}[\cite{BainbridgeHabeggerMollerHNFilt}]
\label{NoZeroInNonVary}
There are no Teichm\"uller curves in genus three with non-trivial Forni subspace in any stratum other than possibly the principal stratum.
\end{theorem}

Since all rank one affine manifolds in $\cH(4)$ are Teichm\"uller curves, we record the following trivial corollary.

\begin{corollary}
\label{NoZeroTeichCurvesInH4}
There are no rank one affine manifolds in $\cH(4)$ with non-trivial Forni subspace.
\end{corollary}

\begin{lemma}
\label{Gen3Rank2ImpQ}
All rank two affine manifolds in genus three are arithmetic.
\end{lemma}

\begin{proof}
By the inequality in \cite[Thm. 1.5]{WrightFieldofDef}, the product of the rank of the affine manifold and the degree of the field of definition is bounded by three.  Therefore, the degree of the field of definition must be one.
\end{proof}

The following lemma is well-known, but a formal proof can be found in \cite{AulicinoCompDegKZAIS}.

\begin{lemma}\cite{AulicinoCompDegKZAIS}
\label{QImpliesTeichCurves}
Any affine manifold with rational field of definition must contain infinitely many arithmetic Teichm\"uller curves.
\end{lemma}

\begin{theorem}
\label{NoAISLowerStrata}
There are no rank two affine manifolds in genus three with non-trivial Forni subspace outside of the principal stratum.
\end{theorem}

\begin{proof}
Since every rank two affine manifold with non-trivial Forni subspace must contain a Teichm\"uller curve by Lemma \ref{QImpliesTeichCurves}, and the Teichm\"uller curve must also have non-trivial Forni subspace by Lemma \ref{ForniDimSubMan}, there can be no such rank two affine manifolds by Theorem \ref{NoZeroInNonVary}.
\end{proof}

The following theorem is proven at the end of Section \ref{Rank2PrinStrSect} using all of the results in that section.

\begin{theorem}
\label{NoRank2PrinStratum}
There are no rank two affine manifolds with non-trivial Forni subspace in $\mathcal{H}(1,1,1,1)$.
\end{theorem}

The final ingredient needed to prove the main theorem concerns rank one affine manifolds.  In order to prove this theorem, it will be necessary to understand deformations in an orbit closure when all of the absolute periods are fixed.  This is the motivation for the results in Section \ref{RELDefSect}.  Once these are established, they can be applied to prove Theorem \ref{Rank1PrinStratum}, and the proof can be found at the end of Section \ref{Rank1PrinStrSect}.

\begin{theorem}
\label{Rank1PrinStratum}
If $\mathcal{M}$ is a rank one affine manifold with non-trivial Forni subspace in $\Omega\cM_3$, then $\mathcal{M}$ is a Teichm\"uller curve in $\mathcal{H}(1,1,1,1)$.  Equivalently, there are no rank one affine manifolds with non-trivial REL\footnote{See the definition in Section \ref{RELDefSect}.} and non-trivial Forni subspace in $\Omega\cM_3$, and any Teichm\"uller curves with non-trivial Forni subspace lie in the principal stratum.
\end{theorem}

The following lemma was proved in greater generality than is needed here in \cite{AulicinoCompDegKZ, AulicinoCompDegKZAIS}.  In this context we consider sequences of genus three surfaces converging to a genus three surface, so that at the level of the moduli space of Riemann surfaces (forgetting the bundle of Abelian differentials), all of the sequences converge in a compact subset of the space.  However, in the bundle of Abelian differentials, it is possible for zeros to collide and the stratum to change.

\begin{lemma}
\label{DetLocusClosed}
If an affine manifold $\mathcal{M} \subset \mathcal{H}(\kappa)$ has a non-trivial Forni subspace, and the closure of $\mathcal{M}$ contains an affine manifold $\mathcal{M}' \subset \mathcal{H}(\kappa')$ in a boundary stratum in the same genus as $\mathcal{H}(\kappa)$, then $\mathcal{M}'$ also has a non-trivial Forni subspace.
\end{lemma}

\begin{proof}
Since the genus is preserved under this degeneration, for every element $(X', \omega') \in \cM'$ and for all $(X, \omega) \in \cM$ in a neighborhood of $(X', \omega')$, there is a natural identification by the Gauss-Manin connection of the spaces $H^1(X', \bR)$ and $H^1(X, \bR)$.  This holds because both $X$ and $X'$ lie in a compact subset of the moduli space of Riemann surfaces.  Let $G_t^{KZ}$ denote the KZ-cocycle.  Let $\eta_1, \eta_2 \in F \subset H^1(X, \bR)$.  Under the identification, let $\eta'_1, \eta'_2 \in H^1(X', \bR)$ be the elements corresponding to $\eta_1$ and $\eta_2$, respectively.  The inner product can be taken to be either the Hodge inner product, which is defined later, or the symplectic inner product.  Both will work in this case by \cite{AvilaEskinMollerForniBundle}.  Given $\varepsilon > 0$, we can take $X$ sufficiently close to $X'$ so that $\eta_1, \eta_2 \in F(X)$ satisfies
$$|\langle \eta_1, \eta_2 \rangle - \langle \eta'_1, \eta'_2 \rangle| < \varepsilon.$$
By the triangle inequality and the fact that the KZ-cocycle is isometric on elements of $F$, for all $T \geq 0$,
$$|\langle \eta'_1, \eta'_2 \rangle - G_T^{KZ} \cdot \langle \eta'_1, \eta'_2 \rangle | \leq $$
$$|\langle \eta'_1, \eta'_2 \rangle - \langle \eta_1, \eta_2 \rangle| + |G_T^{KZ} \cdot \langle \eta_1, \eta_2 \rangle  - G_T^{KZ} \cdot \langle \eta'_1, \eta'_2 \rangle | < \varepsilon + Me^T \varepsilon,$$
where $M$ depends only on the genus.  Since $\varepsilon > 0$ can be taken arbitrarily small, we have that $\eta'_1, \eta'_2 \in F' \subset H^1(X', \bR)$.  This argument can be applied to a basis of $F$ so that the lemma follows.
\end{proof}

\begin{theorem}
\label{MainThmForniSp}
Let $\mathcal{M}$ be an affine manifold in the moduli space of Abelian differentials on genus three surfaces.  If the Forni subspace of $\mathcal{M}$ is $2$-dimensional, then $\mathcal{M}$ is a Teichm\"uller curve in the principal stratum $\mathcal{H}(1,1,1,1)$ with trace field of degree at most two over the rationals.  Furthermore, there are at most finitely many such Teichm\"uller curves.
\end{theorem}

\begin{proof}
By Theorem \ref{NoZeroInNonVary}, there are no Teichm\"uller curves with non-trivial Forni subspace outside of the principal stratum.  By Theorem \ref{Rank1PrinStratum}, there are no rank one affine manifolds with non-trivial REL and non-trivial Forni subspace.  By Theorems \ref{NoAISLowerStrata} and \ref{NoRank2PrinStratum}, there are no rank two affine manifolds with non-trivial Forni subspace.  Since \cite{AvilaEskinMollerForniBundle} implies that the Forni subspace is orthogonal to the projection of the tangent subspace into absolute homology, there can be no rank three affine manifolds with non-trivial Forni subspace.

Since the degree of the trace field is bounded by the genus, it suffices to show that the trace field cannot be cubic.  A Teichm\"uller curve in genus three with cubic trace field is called algebraically primitive.  By \cite[Cor. 3]{ForniCriterion}, all algebraically primitive Teichm\"uller curves have no zero Lyapunov exponents.  Hence, the Forni subspace is trivial.

By contradiction, if there were infinitely many Teichm\"uller curves with non-trivial Forni subspace, then they would equidistribute to a higher dimensional affine manifold by \cite{EskinMirzakhaniMohammadiOrbitClosures}, and the higher dimensional orbit closure would also have non-trivial Forni subspace because $\mathcal{D}_g$ is closed.  Since no such higher dimensional affine manifold exists by the aforementioned theorems, there cannot be infinitely many such Teichm\"uller curves.
\end{proof}

\section{REL Cylinder Deformations}
\label{RELDefSect}

In this section, we will consider rank one affine manifolds that are not Teichm\"uller curves.  The results in this section do not make any assumptions about a specific genus or the dimension of a Forni subspace.  In order to take advantage of the extra dimensions of the affine manifold (those beyond the dimensions of a Teichm\"uller curve), it is essential for us to describe deformations of a translation surface that cannot be described directly by \splin , yet allow us to remain in the affine manifold.  This was accomplished in \cite{WrightCylDef} when he described how to deform cylinders while remaining in a given affine manifold.  The results below follow immediately from his paper as evinced by their proofs.

We begin by recalling perhaps the most important fact about rank one affine manifolds.  Recall that a translation surface is \emph{completely periodic} if every direction that admits a cylinder fully decomposes into cylinders.  The following theorem first appeared in \cite[Thm. 1.5]{WrightCylDef} and a simpler proof was given in \cite[Prop. 4.20]{WrightSurvey}.

\begin{theorem}\cite{WrightCylDef, WrightSurvey}
\label{QuadImpCP}
Let $\mathcal{M}$ be an affine manifold.  If $\text{Rank}(\mathcal{M}) = 1$, then every translation surface in $\mathcal{M}$ is completely periodic.
\end{theorem}

\begin{definition}
Let $p: H^1(X, \Sigma, \cdot) \rightarrow H^1(X, \cdot)$ be the canonical projection from relative to absolute cohomology and $T(\mathcal{M}) \subset H^1(X, \Sigma)$ be the tangent space to an affine manifold $\mathcal{M}$.  Then $\mathcal{M}$ has \emph{non-trivial REL} if
$$\dim \ker p |_{T(\mathcal{M})} \not= 0.$$
\end{definition}

The first observation is that \cite[Thm. 1.1]{WrightCylDef}, which relates the rank of an affine manifold to admissible cylinder deformations, is trivial for rank one manifolds because the cylinder deformations are exactly given by the usual action by \splin .  For a Veech surface, this does indeed describe all possible deformations that preserve its Teichm\"uller curve.  On the other hand, in the presence of non-trivial REL, it is often possible to deform cylinders while fixing all absolute periods.  The goal of this section is to prove that in a rank one affine manifold with non-trivial REL, there always exists a translation surface on which we can perform cylinder deformations analogous to the ones that can be performed on translation surfaces in higher rank manifolds.  Recall the definition of a cylinder deformation introduced in \cite[Sect. 3]{WrightCylDef}.  The deformation in his paper generalizes a cylinder twist.  We also define the analogous generalization of a cylinder stretch.  Let
$$u_t = \left(\begin{array}{cc} 1 & t\\0 & 1 \end{array}\right), \quad a_s = \left(\begin{array}{cc} 1 & 0\\0 & e^s \end{array}\right).$$

\begin{definition}
Let $M$ be a translation surface that decomposes into horizontal cylinders $\mathcal{C} = \{C_1, \ldots, C_r\}$ such that $M = \cup_i C_i$.  Define the \emph{generalized cylinder twist} $u_{t_1, \ldots, t_r}^{\mathcal{C}} = u_{\tau}^{\mathcal{C}}$ to be a deformation given by multiplying $C_i$ by $u_{t_i}$, for all $i$, where $t_i \in \mathbb{R}$.  Similarly, define the \emph{generalized cylinder stretch} $a_{t_1, \ldots, t_r}^{\mathcal{C}} = a_{\tau}^{\mathcal{C}}$ to be a deformation given by multiplying $C_i$ by $a_{t_i}$, for all $i$, where $t_i \in \mathbb{R}$.
\end{definition}

For the sake of brevity, we call a generalized cylinder twist (resp. generalized cylinder stretch) of $M$ a \emph{REL twist} (resp. \emph{REL stretch}) of $M$ if it fixes all of the absolute periods of $M$.  This will suffice for the remainder of this paper.  Finally, we recall the definitions of two important subspaces of the real tangent space to an affine manifold.

\begin{definition}
Let $M \in \mathcal{M}$ be horizontally periodic.  Define the \emph{twist space} of $\mathcal{M}$ at $M$, $\text{Twist}(M, \mathcal{M}) \subset T_{\mathbb{R}}(\mathcal{M})$, to be the subspace of cohomology classes in $T_{\mathbb{R}}(\mathcal{M})$ that are zero on all horizontal saddle connections.

Define the \emph{cylinder preserving space} of $\mathcal{M}$ at $M$ to be the subspace $\text{Pres}(M, \mathcal{M}) \subset T_{\mathbb{R}}(\mathcal{M})$ of cohomology classes that are zero on the core curves of all horizontal cylinders.
\end{definition}

It is clear from the definitions that $\text{Twist}(M, \mathcal{M}) \subseteq \text{Pres}(M, \mathcal{M})$ always holds.

\begin{lemma}[\cite{WrightCylDef} Lem. 8.6]
\label{WrightLem86}
Let $M \in \mathcal{M}$ be horizontally periodic.  If $\text{Twist}(M, \mathcal{M}) \not= \text{Pres}(M, \mathcal{M})$, then there exists a horizontally periodic surface in $\mathcal{M}$ with more horizontal cylinders than $M$.
\end{lemma}

Throughout this paper the conclusion $\text{Twist}(M, \mathcal{M}) \not= \text{Pres}(M, \mathcal{M})$ will be used interchangeably with the conclusion that there exists a translation surface in $\cM$ with more cylinders than $M$.  A similar argument to the proof of the following theorem occurs in \cite[Pf. of Lem. 3.1]{LanneauNguyenWrightFinInNonArithRkOne}.

\begin{theorem}\footnote{The proof of this theorem was provided by Alex Wright.}
\label{RELDefs}
If $\mathcal{M}$ is a rank one orbit closure with non-trivial REL, then there exists a horizontally periodic translation surface $M \in \mathcal{M}$ such that there exists a linear path in period coordinates of REL twists $u_{\tau}$ satisfying $u_{\tau} \cdot M \in \mathcal{M}$, and a linear path in period coordinates of REL stretches $a_{\tau}$ satisfying $a_{\tau} \cdot M \in \mathcal{M}$.  Moreover, the path of REL stretches can be continued as long as the cylinders persist and the zeros do not collapse.
\end{theorem}

\begin{proof}
By Lemma \ref{WrightLem86}, every affine manifold contains a horizontally periodic translation surface $M$ such that $\text{Twist}(M, \mathcal{M}) = \text{Pres}(M, \mathcal{M})$.  By \cite[Lem. 8.8, Cor. 8.12]{WrightCylDef}, $\dim \text{Twist}(M, \mathcal{M}) = \dim T_{\mathbb{R}}( \mathcal{M}) - 1$.  Since $\dim \text{Twist}(M, \mathcal{M}) \geq 2$, $\dim \text{Twist}(M, \mathcal{M}) \cap \ker(p) \geq \dim T_{\mathbb{R}}( \mathcal{M}) - 2$.  Furthermore, $\dim T_{\mathbb{R}}( \mathcal{M}) - 2 > 0$ because $\mathcal{M}$ has non-trivial REL.  Hence, there exists a non-trivial element $\eta \in \text{Twist}(M, \mathcal{M}) \cap \ker(p)$, which by definition is a twist of the cylinders that fixes absolute periods.
\end{proof}

\begin{corollary}
\label{CoreCurvesInd}
Given a rank one orbit closure $\mathcal{M}$ with non-trivial REL, there exists a horizontally periodic translation surface $M$ decomposing into cylinders with core curves $\Gamma = \{\gamma_1, \ldots, \gamma_r\} \subset H_1(X, \mathbb{Z})$ such that $\Gamma$ forms a \emph{linearly dependent} set in $H_1(X, \mathbb{R})$.  Moreover, if the core curves of the cylinders do form a \emph{linearly independent} subset of $H_1(X, \mathbb{Z})$, then $\text{Twist}(M, \mathcal{M}) \not= \text{Pres}(M, \mathcal{M})$.
\end{corollary}

\begin{proof}
Let $\eta_{\gamma_i} \in H^1(X, \bZ)$ denote the cocycle dual to $\gamma_i$.  By contradiction, assume that $\Gamma$ is linearly independent and $\text{Twist}(M, \mathcal{M}) = \text{Pres}(M, \mathcal{M})$.  Then as we observed in the proof of the previous theorem, $\text{Twist}(M, \mathcal{M}) \cap \ker(p) \not= 0$.  Furthermore, \cite[Cor. 8.3]{WrightCylDef} implies that $\text{Twist}(M, \mathcal{M}) = T_{\bR}^M(\cM) \cap \text{span}_{\bR}(\eta_{\gamma_i})_{i=1}^r$.  By assumption and Poincar\'e duality, $\{\eta_{\gamma_1}, \ldots, \eta_{\gamma_r}\}$ is linearly independent in $H^1(X,\bZ)$.  Thus, $\ker p(\text{Twist}(M, \mathcal{M})) = 0$, and the contradiction proves the result.
\end{proof}

We conclude this section with an example of the results proven above.  The results can already be seen in the principal stratum in genus two.  The following example explicitly describes the REL twist and REL stretch for such a surface.

\begin{example}
By Corollary \ref{CoreCurvesInd}, every rank one orbit closure with non-trivial REL contains a horizontally periodic translation surface with cylinders whose core curves are homologically dependent.  It is easy to see in $\mathcal{H}(1,1)$ that the core curves of the cylinders of every $2$-cylinder diagram are homologically independent.  Therefore, every such orbit closure must contain a translation surface $M$ with three cylinders, the maximum number of cylinders possible in genus two.

Theorem \ref{RELDefs} guarantees that there is a REL twist and a REL stretch that can be performed on $M$.  Each of these deformations are depicted in Figure \ref{Gen2RELDefsExFig}.  We leave it to the reader to check that the absolute periods are indeed fixed and that each of these deformations are as claimed.
\end{example}

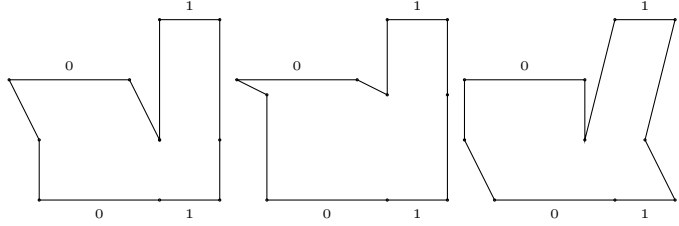
\begin{figure}
\centering
\begin{minipage}{0.24\linewidth}
\centering
\begin{tikzpicture}[scale=0.40]
\draw (0,0)--(0,2)--(-1,4)--(3,4)--(4,2)--(4,6)--(6,6)--(6,0)--cycle;
\draw(1,4) node[above] {\tiny $0$};
\draw(5,6) node[above] {\tiny $1$};
\draw(2,0) node[below] {\tiny $0$};
\draw(5,0) node[below] {\tiny $1$};
\foreach \x in {(0,2),(4,2),(6,2)} \draw \x circle (1pt);
\foreach \x in {(0,0),(-1,4),(3,4),(4,2),(4,6),(6,6),(6,0),(4,0)} \draw \x circle (1pt);
\end{tikzpicture}
\end{minipage}
\begin{minipage}{0.24\linewidth}
\centering
\begin{tikzpicture}[scale=0.40]
\draw (0,0)--(0,3.5)--(-1,4)--(3,4)--(4,3.5)--(4,6)--(6,6)--(6,0)--cycle;
\draw(1,4) node[above] {\tiny $0$};
\draw(5,6) node[above] {\tiny $1$};
\draw(2,0) node[below] {\tiny $0$};
\draw(5,0) node[below] {\tiny $1$};
\foreach \x in {(0,3.5),(4,3.5),(6,3.5)} \draw \x circle (1pt);
\foreach \x in {(0,0),(-1,4),(3,4),(4,6),(6,6),(6,0),(4,0)} \draw \x circle (1pt);
\end{tikzpicture}
\end{minipage}
\begin{minipage}{0.24\linewidth}
\centering
\begin{tikzpicture}[scale=0.40]
\draw (0,0)--(-1,2)--(-1,4)--(3,4)--(3,2)--(4,6)--(6,6)--(5,2)--(6,0)--cycle;
\draw(1,4) node[above] {\tiny $0$};
\draw(5,6) node[above] {\tiny $1$};
\draw(2,0) node[below] {\tiny $0$};
\draw(5,0) node[below] {\tiny $1$};
\foreach \x in {(-1,2),(3,2),(5,2)} \draw \x circle (1pt);
\foreach \x in {(0,0),(-1,4),(3,4),(4,6),(6,6),(6,0),(4,0)} \draw \x circle (1pt);
\end{tikzpicture}
\end{minipage}
 \caption{Examples of a REL stretch (center) and a REL twist (right) of a genus two surface (left)}
 \label{Gen2RELDefsExFig}
\end{figure}

\section{Rank One Affine Manifolds in $\Omega\cM_3$}
\label{Rank1PrinStrSect}

Throughout this section we will assume that the rank of all affine manifolds is one and the Forni subspace is $2$-dimensional.  We will also assume that the rank one affine manifolds have non-trivial REL so as to exclude the case of Teichm\"uller curves.  The goal of this section is to prove Theorem \ref{Rank1PrinStratum}.

\subsection{Cylinder Configurations}

In this section, we introduce the term cylinder configuration, and explain how Forni's Geometric Criterion implies that every periodic translation surface in any genus three orbit closure with non-trivial Forni subspace must satisfy one of very few cylinder configurations.  Cylinder configurations have also proven to be useful in dividing cases in \cite{AulicinoNguyenGen3TwoZeros}.

Given a decomposition of a surface into cylinders, which always occurs in the rank one case by Theorem \ref{QuadImpCP}, the topological data consisting of the horizontal cylinders and the saddle connections in their boundaries is called a \emph{cylinder diagram}.

We consider the following degeneration of a periodic translation surface $M$ that will be very helpful in the discussion that follows.  Consider the degeneration from pinching the core curves of every cylinder.  This is in fact a boundary point of every affine manifold containing $M$ with the Deligne-Mumford compactification by \cite{MasurThesis}.  To see this let $G_t$ be the Teichm\"uller geodesic flow and observe that by \cite[Thm. 3]{MasurThesis},
$$\lim_{t \rightarrow \infty} G_t\cdot M$$
converges to exactly the nodal Riemann surface resulting from pinching the core curves of every cylinder.  Furthermore, every node of the Riemann surface corresponds to a pair of simple poles of an Abelian differential (with opposite residue on each side of the node).  After removing the nodes from this surface, we get a disjoint union of Riemann surfaces with finitely many pairs of punctures where each pair of punctures corresponds to the removal of a node.  A connected component of a degenerate surface is called a \emph{part}.  For example, a surface in genus three with two non-homologous cylinders would degenerate in this sense to a torus with two pairs of simple poles, and this is depicted in Table \ref{MainConfigTable} Configuration 1) Column 2.  On the other hand, a surface in genus three decomposing into two homologous cylinders would degenerate to two tori joined by two pairs of simple poles as in Table \ref{MainConfigTable} Configuration 6) Column 2.

One very useful relation to recall for any meromorphic differential on a genus $g$ Riemann surface is
$$\sharp (\text{Zeros}) - \sharp (\text{Poles}) = 2g-2,$$
where $\sharp(\cdot)$ denotes the total order of the elements in the set.

Another key observation is that it is easy to compute the homological dimension of the span of the core curves of the cylinders from the degenerate surface.  Let $g$ be the genus of the non-degenerate surface and let $g_i$ be the genus of the $i$'th part of the degenerate surface.  The homological dimension of the span of the core curves of the cylinders is given by
$$g - \sum_i g_i.$$
For example, in Configuration 6), the homological dimension of the two (homologous) cylinders is certainly $1 = 3-(1+1)$.

\begin{lemma}
\label{Gen3TopConfigs}
Let $M \in \mathcal{M} \subset \Omega\mathcal{M}_3$ be horizontally periodic.  If $\dim F = 2$, then pinching the core curves of every horizontal cylinder on $M$ yields one of the following six degenerate surfaces (up to homeomorphism) pictured in Column 2 of Table \ref{MainConfigTable}.
\end{lemma}

\begin{proof}
Let $M'$ denote the degenerate surface resulting from pinching the core curve of every horizontal cylinder on $M$.  By the Forni Geometric Criterion \cite[Thm. 4]{ForniCriterion}, the homological dimension of the core curves of the cylinders is bounded above by two because the Forni subspace is nontrivial, by assumption.  First, assume that the core curves of the cylinders represent a $2$-dimensional subspace of homology.  After pinching the core curves of the cylinders and removing the resulting nodes, the degenerate surface must have total genus one.  Therefore, $M'$ is a torus and a (possibly empty) collection of spheres.

If there are no spheres attached to the torus, then we get Configuration 1) in Table \ref{MainConfigTable}.

If there is one sphere attached to the torus, there must be at least two cylinders joining the two parts because the core curve of a cylinder of an Abelian differential is never homologous to zero, and therefore, we get Configurations 2) and 3) in Table \ref{MainConfigTable}.

If the degenerate surface consists of two spheres, $S_1$, $S_2$, and a torus, then we know that each sphere carries at least three simple poles, which correspond to three half-infinite cylinders.  The torus itself must have at least two simple poles by the observation above that the core curve of a cylinder for an Abelian differential is never a separating curve.  We consider which parts the nodes must connect in this picture.  First assume by contradiction that the two nodes on the torus are connected to one of the two spheres, say $S_1$.  Then any way of connecting the sphere $S_2$ to the rest of the nodal surface by three nodes will force the genus of $M$ to be strictly greater than three.  This leaves the case of the torus being connected to each of the spheres by a node.  As noted above, $M'$ has at least two more nodes on each sphere and unless they lie between $S_1$ and $S_2$, then $M$ will have genus greater than three.  Thus, Configuration 4) is the only degenerate surface that satisfies this.

It is impossible to have three spheres attached to the torus in genus three.  All three spheres must have at least three or more half cylinders, and the torus, as before, must have at least two half cylinders.  However, genus three surfaces have at most six cylinders, and in that case the core curves of the cylinders span a Lagrangian subspace of homology.  Hence, $M'$ cannot have three spheres.

If the homological dimension of the cylinders is one, then there are only two possibilities as depicted in 5) and 6).  They correspond to a single cylinder and two homologous cylinders, respectively.  It is impossible to have three homologous cylinders on a genus three surface.
\end{proof}

\begin{definition}
Two cylinder diagrams are \emph{equivalent} if after pinching the core curves of every cylinder in both diagrams, the resulting degenerate surfaces are homeomorphic.  Define a \emph{cylinder configuration} to be an equivalence class of cylinder diagrams under this definition of equivalence.
\end{definition}

By definition every degenerate surface gives rise to a cylinder configuration, and these were completely described in Lemma \ref{Gen3TopConfigs}.  Obviously, each cylinder configuration imposes significant restrictions on the cylinder diagrams contained in it.  Below we make some general observations about every cylinder diagram contained in each cylinder configuration.  In Column 3 of Table \ref{MainConfigTable}, we provide figures of cylinder diagrams corresponding to the observations we make below.  In the figures, $\cE$ denotes a collection of saddle connections, which in some cases (noted below) are permitted to be the empty set.  For certain cylinder configurations it will be necessary for our arguments to list every cylinder diagram in that cylinder configuration.  However, we postpone this to later sections.

\

\noindent \textit{Configuration 1}:  There are two non-homologous cylinders in this configuration and at least one cylinder has a saddle connection contained in both its top and bottom.  In the figure, the collection of saddle connections $\mathcal{E}_1$ is always non-empty after renumbering.  However, we permit $\mathcal{E}_3$ to be the empty set.

\

\noindent \textit{Configuration 2}: There are three cylinders in this configuration and since one of the parts of the degenerate surface is a sphere with three poles, there must be exactly one simple zero on that part.  Moreover, there is exactly one way to identify three (half) cylinders so that they form a sphere with a simple zero.  The tops of $C_1$ and $C_2$ must be identified to the bottom of $C_3$.

\

\noindent \textit{Configuration 3}: There are exactly three cylinders in this configuration.  Since one of the parts is a torus with two poles, the two cylinders, say $C_2$ and $C_3$, incident with the torus are homologous.  In particular, the circumferences of $C_2$ and $C_3$ are equal.  Furthermore, the top of one of these cylinders, say $C_3$, must be entirely identified to the bottom of the other cylinder $C_2$ in order for the degeneration by pinching core curves to form a torus.  The third cylinder $C_1$ is identified to $C_2$ and $C_3$ in some manner that will be addressed later, cf. Lemma \ref{CylDiags3}.  The figure in Table \ref{MainConfigTable} only depicts one of three very different identifications!

\

\noindent \textit{Configuration 4}: This configuration consists of four cylinders and the justification for the arrangement is a combination of the arguments above.  The torus again must be formed by identifying the top of one cylinder, say $C_2$, to the bottom of another cylinder, say $C_1$.  All other identifications are among three cylinders that identify along a simple zero.  There is exactly one possible identification among a choice of three cylinders as noted in Configuration 2) above.  However, there are two possible cylinder diagrams resulting from the two possible choices in this case, cf. Figure \ref{Config4CylDiagFig}, and this will be rigorously addressed in Lemma \ref{CylDiags4}.

\

\noindent \textit{Configuration 5}: Obvious.

\

\noindent \textit{Configuration 6}: Since $C_1$ and $C_2$ are homologous cylinders, the top of the $C_1$ must be entirely identified with the bottom of $C_2$ and vice versa.  Likewise, this must be true for the bottom of $C_1$ and the top of $C_2$.  If not, the core curves of each of $C_1$ and $C_2$ would not be homologous.\footnote{In fact, there is a unique cylinder diagram in this case in $\cH(1,1,1,1)$, and every periodic direction on the Eierlegende Wollmilchsau satisfies it.}

\begin{table}
$$\begin{array}{c|c|c}
\text{Configuration} & \text{Degeneration Figure} & \text{Cylinder Configuration} \\
\hline
1) & \begin{minipage}[c]{0.24\linewidth}
\centering
\begin{tikzpicture}[thick,scale=0.5]
\draw[thick] (0,0) circle(2.5);
\draw plot [smooth] coordinates {(-2,.2) (-1.5,-.2) (-1,.2)};
\draw plot [smooth] coordinates {(-1.75,0) (-1.5,.2) (-1.25,0)};
\draw {(-1.2,.9) node{-1}[circle, draw, fill=black,inner sep=0pt, minimum width=5pt] -- (1.2,.9) node{-1}[circle, draw, fill=black!50,inner sep=0pt, minimum width=5pt]};
\draw {(-1.2,-.9) node{-1}[circle, draw, fill=black,inner sep=0pt, minimum width=5pt] -- (1.2,-.9) node{-1}[circle, draw, fill=black!50,inner sep=0pt, minimum width=5pt]};
\end{tikzpicture}
\end{minipage} & \begin{minipage}[c]{0.24\linewidth}
\centering
\begin{tikzpicture}[scale=0.30]
\draw (0,0)--(0,2)--(-3,2)--(-3,4)--(4,4)--(4,2)--(6,2)--(6,0)--cycle;
\draw(-1,4) node[above] {\tiny $\mathcal{E}_1$};
\draw(3,4) node[above] {\tiny $\mathcal{E}_2$};
\draw(5,2) node[above] {\tiny $\mathcal{E}_3$};
\draw(-2,2) node[below] {\tiny $\mathcal{E}_1$};
\draw(1,0) node[below] {\tiny $\mathcal{E}_2$};
\draw(4,0) node[below] {\tiny $\mathcal{E}_3$};
\end{tikzpicture}
\end{minipage} \\
2) & \begin{minipage}[c]{0.24\linewidth}
\centering
\begin{tikzpicture}[thick,scale=0.4]
\draw[thick] (2,0) circle(1.5);
\draw[thick] (-2,0) circle(1.5);
\draw plot [smooth] coordinates {(2.5,.2) (2,-.2) (1.5,.2)};
\draw plot [smooth] coordinates {(2.25,0) (2,.2) (1.75,0)};
\draw {(-1.2,.9) node{-1}[circle, draw, fill=black,inner sep=0pt, minimum width=5pt] -- (1.2,.9) node{-1}[circle, draw, fill=black!50,inner sep=0pt, minimum width=5pt]};
\draw {(-1.2,0) node{-1}[circle, draw, fill=black,inner sep=0pt, minimum width=5pt] -- (1.2,0) node{-1}[circle, draw, fill=black!50,inner sep=0pt, minimum width=5pt]};
\draw {(-1.2,-.9) node{-1}[circle, draw, fill=black,inner sep=0pt, minimum width=5pt] -- (1.2,-.9) node{-1}[circle, draw, fill=black!50,inner sep=0pt, minimum width=5pt]};
\end{tikzpicture}
\end{minipage} & \begin{minipage}[c]{0.24\linewidth}
\centering
\begin{tikzpicture}[scale=0.30]
\draw (0,0)--(0,4)--(1.9,4)--(2,2)--(2.1,4)--(5,4)--(5,0)--cycle;
\foreach \x in {(0,2),(2,2),(5,2)} \draw \x circle (1pt);
\draw(1,3) node {\tiny $C_1$};
\draw(3.5,3) node {\tiny $C_2$};
\draw(2.5,1) node {\tiny $C_3$};
\draw(1,4) node[above] {\tiny $\mathcal{E}_1$};
\draw(4,4) node[above] {\tiny $\mathcal{E}_2$};
\draw(1,0) node[below] {\tiny $\mathcal{E}_1$};
\draw(4,0) node[below] {\tiny $\mathcal{E}_2$};
\end{tikzpicture}
\end{minipage}  \\
3) & \begin{minipage}[c]{0.24\linewidth}
\centering
\begin{tikzpicture}[thick,scale=0.4]
\draw[thick] (2,0) circle(1.5);
\draw[thick] (-2,0) circle(1.5);
\draw plot [smooth] coordinates {(2.5,.2) (2,-.2) (1.5,.2)};
\draw plot [smooth] coordinates {(2.25,0) (2,.2) (1.75,0)};
\draw {(-3,-.5) node{-1}[circle, draw, fill=black,inner sep=0pt, minimum width=5pt] -- (-3,.5) node{-1}[circle, draw, fill=black!50,inner sep=0pt, minimum width=5pt]};
\draw {(-1.2,.9) node{-1}[circle, draw, fill=black,inner sep=0pt, minimum width=5pt] -- (1.2,.9) node{-1}[circle, draw, fill=black!50,inner sep=0pt, minimum width=5pt]};
\draw {(-1.2,-.9) node{-1}[circle, draw, fill=black,inner sep=0pt, minimum width=5pt] -- (1.2,-.9) node{-1}[circle, draw, fill=black!50,inner sep=0pt, minimum width=5pt]};
\end{tikzpicture}
\end{minipage} & \begin{minipage}[c]{0.24\linewidth}
\centering
\begin{tikzpicture}[scale=0.30]
\draw (0,0)--(0,2)--(-2,2)--(-2,4)--(4,4)--(4,2)--(6,2)--(6,0)--(6,-2)--(4,-2)--(4,0)--cycle;
\draw(5,-1) node {\tiny $C_1$};
\draw(1,3) node {\tiny $C_2$};
\draw(3,1) node {\tiny $C_3$};
\draw(-1,4) node[above] {\tiny $\mathcal{E}_2$};
\draw(3,4) node[above] {\tiny $\mathcal{E}_3$};
\draw(5,2) node[above] {\tiny $\mathcal{E}_1$};
\draw(-1,2) node[below] {\tiny $\mathcal{E}_1$};
\draw(1,0) node[below] {\tiny $\mathcal{E}_2$};
\draw(5,-2) node[below] {\tiny $\mathcal{E}_3$};
\end{tikzpicture}
\end{minipage} \\
4) & \begin{minipage}[c]{0.24\linewidth}
\centering
\begin{tikzpicture}[thick,scale=0.25]
\draw[thick] (-4,0) circle(1.5);
\draw[thick] (0,0) circle(1.5);
\draw[thick] (4,0) circle(1.5);
\draw plot [smooth] coordinates {(-.5,.2) (0,-.2) (.5,.2)};
\draw plot [smooth] coordinates {(-.25,0) (0,.2) (.25,0)};
\draw  {(-4,.9) node{-1}[circle, draw, fill=black,inner sep=0pt, minimum width=5pt] -- (0,2)--(4,.9) node{-1}[circle, draw, fill=black!50,inner sep=0pt, minimum width=5pt]}[smooth];
\draw [smooth] {(-4,-.9) node{-1}[circle, draw, fill=black,inner sep=0pt, minimum width=5pt] -- (0,-2)--(4,-.9) node{-1}[circle, draw, fill=black!50,inner sep=0pt, minimum width=5pt]};
\draw {(-3,0) node{-1}[circle, draw, fill=black,inner sep=0pt, minimum width=5pt] -- (-1,0) node{-1}[circle, draw, fill=black!50,inner sep=0pt, minimum width=5pt]};
\draw {(1,0) node{-1}[circle, draw, fill=black,inner sep=0pt, minimum width=5pt] -- (3,0) node{-1}[circle, draw, fill=black!50,inner sep=0pt, minimum width=5pt]};
\end{tikzpicture}
\end{minipage}  & \begin{minipage}[c]{0.24\linewidth}
\centering
\begin{tikzpicture}[scale=0.30]
\draw (0,0)--(0,4)--(1.9,4)--(2,2)--(2.1,4)--(5,4)--(5,0)--(8,0)--(8,-2)--(3,-2)--(3,0)--cycle;
\foreach \x in {(0,2),(2,2),(5,2),(3,-2),(6,-2),(8,-2)} \draw \x circle (1pt);
\draw(1,3) node {\tiny $C_1$};
\draw(3.5,3) node {\tiny $C_2$};
\draw(2.5,1) node {\tiny $C_3$};
\draw(5.5,-1) node {\tiny $C_4$};
\draw(1,4) node[above] {\tiny 0};
\draw(3.5,4) node[above] {\tiny 1};
\draw(4.5,-2) node[below] {\tiny 1};
\draw(7,-2) node[below] {\tiny 0};
\draw(7,0) node[above] {\tiny $\mathcal{E}_1$};
\draw(1,0) node[below] {\tiny $\mathcal{E}_1$};
\end{tikzpicture}
\end{minipage} \\
5) & \begin{minipage}[c]{0.24\linewidth}
\centering
\begin{tikzpicture}[thick,scale=0.50]
\draw plot [smooth cycle] coordinates {(-2,0) (-1,2) (0,1) (1,2) (4,0) (1,-2) (0,-1) (-1,-2)};
\draw plot [smooth] coordinates {(-1.5,.2) (-1,-.2) (-.5,.2)};
\draw plot [smooth] coordinates {(1.5,.2) (1,-.2) (.5,.2)};
\draw plot [smooth] coordinates {(-1.25,0) (-1,.2) (-.75,0)};
\draw plot [smooth] coordinates {(1.25,0) (1,.2) (.75,0)};
\draw {(3,-.5) node{-1}[circle, draw, fill=black,inner sep=0pt, minimum width=5pt] -- (3,.5) node{-1}[circle, draw, fill=black!50,inner sep=0pt, minimum width=5pt]};
\end{tikzpicture}
\end{minipage} & \begin{minipage}[c]{0.24\linewidth}
\centering
\begin{tikzpicture}[scale=0.30]
\draw (0,0)--(0,2)--(6,2)--(6,0)--cycle;
\draw(3,2) node[above] {\tiny $\mathcal{E}_1$};
\draw(3,0) node[below] {\tiny $\mathcal{E}_1$};
\end{tikzpicture}
\end{minipage} \\
6) & \begin{minipage}[c]{0.24\linewidth}
\centering
\begin{tikzpicture}[thick,scale=0.4]
\draw[thick] (2,0) circle(1.5);
\draw[thick] (-2,0) circle(1.5);
\draw plot [smooth] coordinates {(-2.5,.2) (-2,-.2) (-1.5,.2)};
\draw plot [smooth] coordinates {(-2.25,0) (-2,.2) (-1.75,0)};
\draw plot [smooth] coordinates {(2.5,.2) (2,-.2) (1.5,.2)};
\draw plot [smooth] coordinates {(2.25,0) (2,.2) (1.75,0)};
\draw {(-1.2,.9) node{-1}[circle, draw, fill=black,inner sep=0pt, minimum width=5pt] -- (1.2,.9) node{-1}[circle, draw, fill=black!50,inner sep=0pt, minimum width=5pt]};
\draw {(-1.2,-.9) node{-1}[circle, draw, fill=black,inner sep=0pt, minimum width=5pt] -- (1.2,-.9) node{-1}[circle, draw, fill=black!50,inner sep=0pt, minimum width=5pt]};
\end{tikzpicture}
\end{minipage} & \begin{minipage}[c]{0.24\linewidth}
\centering
\begin{tikzpicture}[scale=0.30]
\draw (0,0)--(0,2)--(-2,2)--(-2,4)--(4,4)--(4,2)--(6,2)--(6,0)--cycle;
\draw(5,2) node[above] {\tiny $\mathcal{E}_2$};
\draw(-1,2) node[below] {\tiny $\mathcal{E}_2$};
\draw(3,0) node[below] {\tiny $\mathcal{E}_1$};
\draw(1,4) node[above] {\tiny $\mathcal{E}_1$};
\end{tikzpicture}
\end{minipage}  \\
\end{array}$$
\caption{Examples of Cylinder Configurations for Degenerate Surfaces: $\mathcal{E}_i$ denotes a collection of saddle connections and $-1$ indicates a simple pole of an Abelian differential}
\label{MainConfigTable}
\end{table}

\subsection{Reduction to Configurations 2), 3), 4)}

The goal of this subsection is to prove Lemma \ref{3PlusCyls}, which will dramatically reduce the combinatorics necessary to prove Theorem \ref{Rank1PrinStratum}.

\begin{lemma}
\label{SC2Cyls}
If a translation surface is the union of two non-homologous cylinders, then at least one of them has a saddle connection $\sigma$ on its top and bottom.
\end{lemma}

\begin{proof}
If not, then the two cylinders would be homologous.
\end{proof}

Recall that a \emph{simple cylinder} is one whose boundaries each consist of a single saddle connection.

\begin{lemma}
\label{3PlusCyls}
If $\mathcal{M}$ is a rank one affine manifold with non-trivial REL in $\Omega\mathcal{M}_3$, then $\mathcal{M}$ contains a horizontally periodic translation surface with at least three cylinders.
\end{lemma}

\begin{proof}
First, assume that every cylinder decomposition of a periodic surface in $\cM$ decomposes into two homologous cylinders, i.e. Configuration 6).  By contradiction, if no translation surface in $\mathcal{M}$ admits any other cylinder decomposition, then Theorem \ref{RELDefs} guarantees that we can perform a REL twist on the two cylinders so that the top of $C_1$ and the bottom of $C_2$ are fixed while the two cylinders are twisted against each other.  Let $\sigma$ be a saddle connection on the top and bottom of $C_1$ and $C_2$, respectively.  Let $\sigma'$ be a saddle connection on the bottom of $C_1$ and the top of $C_2$, respectively.  Then we can move $\sigma'$ by a REL twist to any location we like.  Therefore, there exists a REL twist such that there is a closed regular trajectory from $\sigma$ to itself that passes through $\sigma'$ and yields a vertical cylinder $C$.  The circumference of $C$ is equal to the sum of the heights of $C_1$ and $C_2$.  By inspection and complete periodicity, there is a cylinder parallel to $C$ with a different circumference, which implies that the cylinder configuration cannot be Configuration 6).  By inspection of Table \ref{MainConfigTable}, we see that it must satisfy one of Configurations 1) through 4).

Next assume that $M$ has Configurations 1) or 5).  In Configuration 5), every saddle connection occurs on the top and bottom of the cylinder and so there is a foliation transverse to the horizontal foliation that contains a simple cylinder and it is periodic because $M$ is completely periodic by Theorem \ref{QuadImpCP}.  The only cylinder configurations that admit a simple cylinder are 1) through 4).

Let $M$ be the translation surface in Configuration 1).  Since the core curves of the cylinders represent homologically independent classes of curves, Corollary \ref{CoreCurvesInd} implies that $\text{Twist}(M, \mathcal{M}) \not= \text{Pres}(M, \mathcal{M})$.  Therefore, by Lemma \ref{WrightLem86}, there exists a perturbation of $M$ producing a translation surface with more than two cylinders.
\end{proof}

\subsection{Some Technical Lemmas}

We recall the notation of \cite[$\S$ 3]{AthreyaBufetovEskinMirzakhani}.  Let $\omega_1, \omega_2$ be two holomorphic $1$-forms on $X$.  Define the Hodge inner product on the space of $1$-forms on $X$ by 
$$\langle \omega_1, \omega_2 \rangle = \int_X \omega_1 \wedge \overline{\omega_2}.$$
Let $r$ be the natural map from first cohomology to the space of $1$-forms.  This induces an inner product on $H^1(X, \bR)$.  If $\lambda_1, \lambda_2 \in H^1(X, \bR)$, then
$$\langle \lambda_1, \lambda_2 \rangle = \langle r(\lambda_1), r(\lambda_2) \rangle = \int_X \lambda_1 \wedge \star \lambda_2,$$
where $\star$ denotes the Hodge star operator.  Naturally, this inner product induces a norm, known as the Hodge norm and denoted $\|\cdot \|$.  Let $\gamma \in H_1(X, \bR)$.  Define the cohomology class $\star c_{\gamma} \in H^1(X, \bR)$ so that for all $\omega \in H^1(X, \bR)$,
$$\int_{\gamma} \omega = \int_X \omega \wedge \star c_{\gamma}.$$
For $\gamma \in H_1(X, \bR)$ and $\lambda \in H^1(X, \bR)$, define $\lambda(\gamma) = \int_{\gamma}\lambda$ to be the usual pairing.

The following lemma is completely general and does not make any assumptions about the rank of the orbit closure.

\begin{lemma}
\label{ForniVanCyls}
Let $\gamma$ be the core curve of a cylinder $C$ on a translation surface $M = (X, \omega) \in \mathcal{M}$.  Let $F$ be the Forni subspace of $\mathcal{M}$, and let $\eta \in F \subset H^1(X, \mathbb{R})$.  Then $\eta(\gamma) = 0$.
\end{lemma}

\begin{proof}
By the definition of the Forni subspace, the KZ-cocycle is isometric with respect to the Hodge metric on $F$.  In particular, this implies that the Hodge norm is constant on all elements in $F$.

Let $\pi: H^1(X) \rightarrow F$ denote the projection to the Forni subspace.  If $\pi(c_{\gamma}) = 0$, then we are done.  Otherwise, $\langle \eta, \pi(c_{\gamma}) \rangle = A$, where $A$ is constant over each translation surface in $\cM$. By \cite[Thm. 1.5]{AvilaEskinMollerForniBundle} and the Cauchy-Schwarz inequality, we have
$$|A| = |\langle \eta, \pi(c_{\gamma}) \rangle| = |\langle \eta, c_{\gamma} \rangle| \leq \|\eta\| \|c_{\gamma}\|.$$
The relations above imply $\eta(\gamma) = \langle \eta, c_{\gamma} \rangle$.

Since $\gamma$ is the core curve of $C$, use Theorem \ref{WrightCylDefThm} to stretch all cylinders parallel to $C$ so that by the collar lemma, the hyperbolic length of $\gamma$ tends to zero.  By \cite[Thm. 3.1]{AthreyaBufetovEskinMirzakhani}, the Hodge norm of $c_{\gamma}$ tends to zero with the hyperbolic length.  Since the above inequality holds over all $(X, \omega) \in \cM$, we must have $A = 0$.
\end{proof}

\begin{corollary}
\label{HomBasisForniContraCor}
Let $M$ be a horizontally periodic genus three translation surface in a rank one orbit closure $\mathcal{M}$ with non-trivial REL.  Let $M$ admit an absolute homology basis $\{a_1, a_2, a_3, b_1, b_2, b_3\}$.  If
\begin{itemize}
\item For each $\gamma \in \{a_1, a_2, b_1, b_2, b_3\}$, there exists $M_{\gamma} \in \cM$ such that $\gamma$ is the core curve of a cylinder on $M_{\gamma} \in \cM$, or
\item $\cB = \{b_1, b_2, b_3\}$ generates a Lagrangian subspace of the absolute homology space of $M$, and for each $\gamma \in \cB$ there exists $M_{\gamma} \in \cM$ such that $\gamma$ is the core curve of a cylinder on $M_{\gamma}$,
\end{itemize}
then the Forni subspace of $\mathcal{M}$ has dimension zero.
\end{corollary}

\begin{proof}
By Lemma \ref{ForniVanCyls} and the first assumption, every element of the Forni subspace evaluates to zero on a $5$-dimensional subspace of absolute cohomology.  In genus three, absolute cohomology is $6$-dimensional.  By Theorem \ref{ForniConst}, the Forni subspace is symplectic, which implies it must be zero dimensional.

By Lemma \ref{ForniVanCyls} and the second assumption, every element of the Forni subspace evaluates to zero on the three elements of $\cB$.  By Theorem \ref{ForniConst}, the Forni subspace is a symplectic subspace of absolute homology, so it cannot be orthogonal to a Lagrangian subspace of absolute homology unless it is zero dimensional.
\end{proof}

\subsection{Configuration 2)}

\begin{convention}
There will be several cylinder diagrams presented in the remainder of this section.  We introduce a non-standard convention that will be very helpful in proving the existence of trajectories we claim exist.  Instead of drawing a cylinder as a parallelogram to indicate that it is twisted relative to another cylinder in the diagram.  We draw the cylinder as a rectangle and do not assume that there are zeros at the corner of the rectangle as is usually assumed to be the case throughout the literature.  For an example of a cylinder diagram using this convention see Figure \ref{CylDiagConv} for two equivalent diagrams: the left one uses the usual convention and the right one uses the one described here.  We will use a prime to indicate the continuation of a saddle connection from one side of a rectangle to the other.
\end{convention}

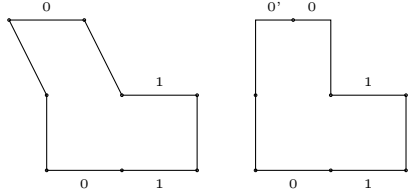
\begin{figure}[ht]
\centering
\begin{minipage}[b]{0.24\linewidth}
\centering
\begin{tikzpicture}[scale=0.50]
\draw (0,0)--(0,2)--(-1,4)--(1,4)--(2,2)--(4,2)--(4,0)--cycle;
\foreach \x in {(0,0),(0,2),(-1,4),(1,4),(2,2),(4,2),(4,0),(2,0)} \draw \x circle (1pt);
\draw(0,4) node[above] {\tiny 0};
\draw(3,2) node[above] {\tiny 1};
\draw(1,0) node[below] {\tiny 0};
\draw(3,0) node[below] {\tiny 1};
\end{tikzpicture}
\end{minipage}
\begin{minipage}[b]{0.24\linewidth}
\centering
\begin{tikzpicture}[scale=0.50]
\draw (0,0)--(0,4)--(2,4)--(2,2)--(4,2)--(4,0)--cycle;
\foreach \x in {(0,0),(0,2),(1,4),(2,2),(4,2),(4,0),(2,0)} \draw \x circle (1pt);
\draw(.5,4) node[above] {\tiny 0'};
\draw(1.5,4) node[above] {\tiny 0};
\draw(3,2) node[above] {\tiny 1};
\draw(1,0) node[below] {\tiny 0};
\draw(3,0) node[below] {\tiny 1};
\end{tikzpicture}
\end{minipage}
 \caption{Cylinder Diagram Convention: Standard Convention (left) and Current Convention (right)}
 \label{CylDiagConv}
\end{figure}

\begin{remark}
We invite the reader to check that if a cylinder diagram in genus three satisfies Configuration 2), then it must lie in
$$\cH(3,1) \cup \cH(2,1,1) \cup \cH(1,1,1,1).$$
This fact will not be used explicitly, but it helps to have a more accurate picture of this case.
\end{remark}

\begin{lemma}
\label{CylDiag2}
Let $\mathcal{M}$ be a rank one affine manifold with non-trivial REL and $2$-dimensional Forni subspace.  If $M \in \mathcal{M}$ satisfies Configuration 2), then either there exists $M' \in \cM$ satisfying Configuration 4) or there exists $\cM'$ in a lower dimensional stratum in genus three such that $\cM'$ has a $2$-dimensional Forni subspace.
\end{lemma}

\begin{proof}
If $\text{Twist}(M, \mathcal{M}) \not= \text{Pres}(M, \mathcal{M})$, then by Lemma \ref{WrightLem86}, there exists a translation surface $M' \in \cM$ with more cylinders.  However, $\cM$ has a $2$-dimensional Forni subspace, and the only cylinder configuration with more than three cylinders is Configuration 4) by Lemma \ref{Gen3TopConfigs}.  Thus, it suffices to assume that $\text{Twist}(M, \mathcal{M}) = \text{Pres}(M, \mathcal{M})$, and Theorem \ref{RELDefs} implies $M$ can be deformed by REL stretches and REL twists.

First we consider the case where $C_1$ and $C_2$ have distinct heights, and perform a REL twist so that the shorter of the two cylinders has a zero on its top lying directly over a zero on its bottom, i.e. the shorter cylinder contains a vertical saddle connection.  We claim that we can collapse the cylinder with a REL stretch and every other cylinder will persist.  Observe that there are absolute homology cycles $b_1$ and $b_2$ crossing the heights of $C_3$, and $C_1$ and $C_2$, respectively.  Let $h_i$ denote the height of $C_i$.  The imaginary parts of the periods of $b_1$ and $b_2$ are $h_1 + h_3$ and $h_2 + h_3$, respectively.  Hence, both of these quantities are fixed under a REL deformation by definition.  Without loss of generality, let $C_1$ be the shorter cylinder, so that $h_1 < h_2$.  Performing a REL stretch that increases the height $h_3$ causes $h_1$ and $h_2$ to decrease.  Since $h_1 < h_2$, the height of $C_1$ will converge to zero, while the height of cylinder $C_2$ becomes $h_2 - h_1 > 0$.  Let $M'$ be the resulting degenerate surface.

We claim that $M'$ is a translation surface of genus three in a lower dimensional stratum than $M$.  To see this it suffices to prove that no curves degenerated to zero length, while a single saddle connection with distinct zeros at each end converged to a point.  Since the resulting surface $M'$ has finite non-zero area, any curve that pinched must be a union of saddle connections on $M_n$.  Therefore, it suffices to show that exactly one saddle connection between distinct zeros degenerated to a point under the REL stretch above.  First observe that all horizontal saddle connections had fixed length under the REL stretch so all horizontal saddle connections persist on $M'$.  Secondly, observe that any saddle connection crossing the height of $C_2$ or $C_3$ also has length bounded away from zero because the heights of the cylinders, which are both non-zero in the limit, are a lower bound for the lengths of saddle connections traversing the heights of $C_2$ and $C_3$.  Therefore, only the length of a saddle connection entirely contained in $C_1$ can converge to zero length.  However, since the horizontal (real) components of periods are fixed under a REL stretch, only a vertical saddle connection contained in $C_1$ can degenerate to a point on $M'$.  Since there is exactly one such saddle connection by construction, no other degeneration occurs.  Finally, by inspection of the cylinder diagram for $M$, we see that the bottom of $C_1$ consists of a single simple zero $v_1$ that only occurs on the bottoms of $C_1$ and $C_2$, and the top of $C_3$.  Hence, every zero on the top of $C_1$ is distinct from $v_1$ and a saddle connection between two distinct zeros can never represent a nonzero element of absolute homology.  By Lemma \ref{DetLocusClosed}, the orbit closure $\cM'$ of $M'$ also has a $2$-dimensional Forni subspace.

Next assume that $C_1$ and $C_2$ have equal heights, and by contradiction that $\text{Twist}(M, \mathcal{M}) = \text{Pres}(M, \mathcal{M})$.  Let $a_1$ and $a_2$ be the core curves of $C_1$ and $C_2$, and let them represent absolute homology cycles.  By assumption, perform a REL twist so that there are no vertical saddle connections and collapse the cylinders $C_1$ and $C_2$ by a REL stretch so that the resulting surface $M'$ also lies in $\cM$ because by the same argument as above no zeros collide.  Then observe that $M'$ will be a horizontal cylinder, and since $M$ and $M'$ have the same genus, $a_1$ and $a_2$ will persist as absolute homology cycles, but obviously not core curves of cylinders.  Since $M'$ is a single cylinder, any choice of $b_1$, $b_2$, and $b_3$ that connect a saddle connection in the bottom of $M'$ to a saddle connection in the top of $M'$ will represent core curves of cylinders.  By reversing the collapse of the cylinders, the curves $b_i$ can still be realized as the core curves of cylinders in a sufficiently small neighborhood of $M'$.  Hence, the assumptions of Corollary \ref{HomBasisForniContraCor} are satisfied, and the resulting contradiction proves $\text{Twist}(M, \mathcal{M}) \not= \text{Pres}(M, \mathcal{M})$.
\end{proof}

\subsection{Configuration 3)}

\begin{remark}
We invite the reader to check that if a cylinder diagram in genus three satisfies Configuration 3), then it must lie in
$$\cH(2,2) \cup \cH(2,1,1) \cup \cH(1,1,1,1).$$
\end{remark}

\begin{lemma}
\label{CylDiags3}
If a cylinder diagram admits Configuration 3), then there are three possible cylinder diagrams in $\cH(1,1,1,1)$ up to reflection and they are depicted in Figures \ref{Config3ACylDiagsFig} and \ref{Config3BCCylDiagsFig}.
\end{lemma}

\begin{proof}
To see this we ``split'' the cylinder diagrams into two parts.  There is the identification along the two homologous cylinders that identify to make the torus, and the rest of the surface.  The identification of the two homologous cylinders depends entirely on the order of the zeros between them because there is exactly one way to identify two cylinders, each having two simple zeros along them, to get a torus.  This follows from the fact that there is a unique $1$-cylinder diagram in $\cH(1,1)$.

Next if we ignore the identification of the two homologous cylinders along a torus and pretend that the homologous cylinders are a single cylinder, we see that such a surface would be a $2$-cylinder diagram in $\cH(1,1)$.  There are exactly two $2$-cylinder diagrams in $\mathcal{H}(1,1)$.  One of the two diagrams in $\cH(1,1)$ is symmetric in the sense that topologically, the cylinders are indistinguishable and therefore, if we replace one of them with the two homologous cylinders, we get Cylinder Diagram 3C).  In the other $2$-cylinder diagram in $\mathcal{H}(1,1)$ we have two choices of cylinders we can pretend were actually resulting from the two homologous cylinders and this yields the other two cylinder diagrams.
\end{proof}

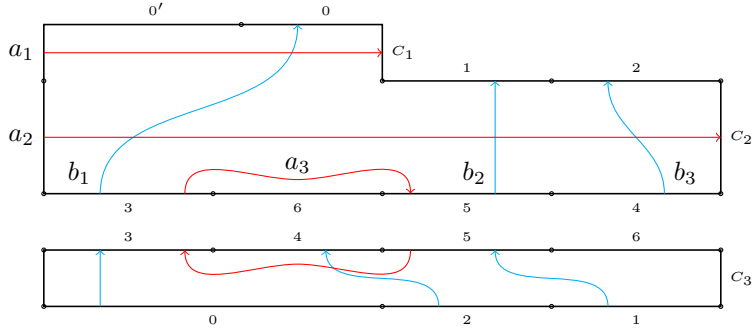
\begin{figure}
 \centering
\begin{tikzpicture}[scale=0.75]
\draw [semithick] (0,1)--(0,4)--(6,4)--(6,3)--(12,3)--(12,1)--cycle;
\draw [semithick] (0,-1)--(0,0)--(12,0)--(12,-1)--cycle;
\foreach \x in {(3.5,4),(0,3),(6,3),(9,3),(12,3),(0,0),(3,0),(6,0),(9,0),(12,0),(0,1),(3,1),(6,1),(9,1),(12,1), (0,-1),(6,-1),(9,-1),(12,-1)} \draw \x circle (1pt);
\draw(6,3.5) node[right] {\tiny $C_1$};
\draw(12,2) node[right] {\tiny $C_2$};
\draw(12,-0.5) node[right] {\tiny $C_3$};
\draw(5,4) node[above] {\tiny $0$};
\draw(2,4) node[above] {\tiny $0'$};
\draw(7.5,3) node[above] {\tiny $1$};
\draw(10.5,3) node[above] {\tiny $2$};
\draw(1.5,1) node[below] {\tiny $3$};
\draw(4.5,1) node[below] {\tiny $6$};
\draw(7.5,1) node[below] {\tiny $5$};
\draw(10.5,1) node[below] {\tiny $4$};
\draw(1.5,0) node[above] {\tiny $3$};
\draw(4.5,0) node[above] {\tiny $4$};
\draw(7.5,0) node[above] {\tiny $5$};
\draw(10.5,0) node[above] {\tiny $6$};
\draw(3,-1) node[below] {\tiny $0$};
\draw(7.5,-1) node[below] {\tiny $2$};
\draw(10.5,-1) node[below] {\tiny $1$};
\draw [->, cyan] (1,-1) -- (1,0);
\draw [->, cyan] (1,1) to [out=90,in=-90] (4.5,4);
\node [above left] at (1,1) {$b_1$};
\draw [->, cyan] (10,-1) to [out=90,in=-90] (8,0);
\draw [->, cyan] (8,1) -- (8,3);
\node [above left] at (8,1) {$b_2$};
\draw [->, cyan] (7,-1) to [out=90,in=-90] (5,0);
\draw [->, cyan] (11,1) to [out=90,in=-90] (10,3);
\node [above right] at (11,1) {$b_3$};
\draw [->, red] (2.5,1) to [out=90,in=180] (4.5,1.25) to [out=0,in=90] (6.5,1);
\draw [<-, red] (2.5,0) to [out=-90,in=180] (4.5,-.25) to [out=0,in=-90] (6.5,0);
\node [above] at (4.5,1.2) {$a_3$};
\draw [->, red] (0,3.5) -- (6,3.5);
\node [left] at (0,3.5) {$a_1$};
\draw [->, red] (0,2) -- (12,2);
\node [left] at (0,2) {$a_2$};

\end{tikzpicture}
 \caption{Cylinder Diagram 3A) with a choice of homology basis}
 \label{Config3ACylDiagsFig}
\end{figure}

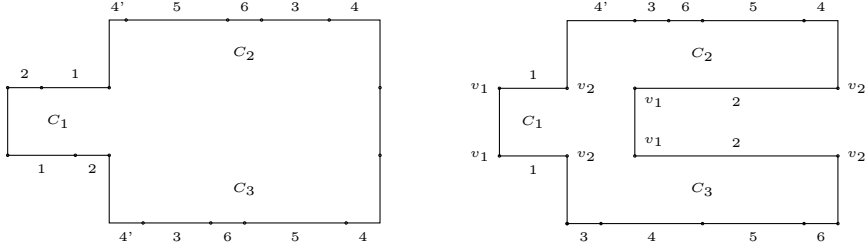
\begin{figure}[h!]
\begin{minipage}[t]{0.24\linewidth}
\centering
\begin{tikzpicture}[scale=0.45]
\draw (0,0)--(0,2)--(-3,2)--(-3,4)--(0,4)--(0,6)--(8,6)--(8,0)--cycle;
\foreach \x in {(0,2),(-1,2),(-3,2),(-3,4),(-2,4),(0,4),(8,4),(8,2)} \draw \x circle (1pt);
\foreach \x in {(1,0),(3,0),(4,0),(7,0)} \draw \x circle (1pt);
\foreach \x in {(.5,6),(3.5,6),(4.5,6),(6.5,6)} \draw \x circle (1pt);
\draw(-1.5,3) node {\tiny $C_1$};
\draw(4,5) node {\tiny $C_2$};
\draw(4,1) node {\tiny $C_3$};
\draw(-2.5,4) node[above] {\tiny 2};
\draw(-1,4) node[above] {\tiny 1};
\draw(-2,2) node[below] {\tiny 1};
\draw(-.5,2) node[below] {\tiny 2};
\draw(.5,0) node[below] {\tiny 4'};
\draw(2,0) node[below] {\tiny 3};
\draw(3.5,0) node[below] {\tiny 6};
\draw(5.5,0) node[below] {\tiny 5};
\draw(7.5,0) node[below] {\tiny 4};
\draw(.25,6) node[above] {\tiny 4'};
\draw(2,6) node[above] {\tiny 5};
\draw(4,6) node[above] {\tiny 6};
\draw(5.5,6) node[above] {\tiny 3};
\draw(7.25,6) node[above] {\tiny 4};
\end{tikzpicture}
\end{minipage}
\begin{minipage}[t]{0.24\linewidth}
\centering
\begin{tikzpicture}[scale=0.30]
\end{tikzpicture}
\end{minipage}
\begin{minipage}[b]{0.24\linewidth}
\centering
\begin{tikzpicture}[scale=0.45]
\draw (0,0)--(0,2)--(-2,2)--(-2,4)--(0,4)--(0,6)--(8,6)--(8,4)--(2,4)--(2,2)--(8,2)--(8,0)--cycle;
\foreach \x in {(0,0),(0,2),(-2,2),(-2,4),(2,2),(2,4),(0,4),(8,4),(8,2)} \draw \x circle (1pt);
\foreach \x in {(0,0),(1,0),(4,0),(7,0),(8,0)} \draw \x circle (1pt);
\foreach \x in {(2,6),(3,6),(4,6),(7,6)} \draw \x circle (1pt);
\draw(-1,3) node {\tiny $C_1$};
\draw(4,5) node {\tiny $C_2$};
\draw(4,1) node {\tiny $C_3$};
\draw(-1,4) node[above] {\tiny 1};
\draw(-1,2) node[below] {\tiny 1};
\draw(5,2) node[above] {\tiny 2};
\draw(5,4) node[below] {\tiny 2};
\draw(.5,0) node[below] {\tiny 3};
\draw(2.5,0) node[below] {\tiny 4};
\draw(5.5,0) node[below] {\tiny 5};
\draw(7.5,0) node[below] {\tiny 6};
\draw(1,6) node[above] {\tiny 4'};
\draw(2.5,6) node[above] {\tiny 3};
\draw(3.5,6) node[above] {\tiny 6};
\draw(5.5,6) node[above] {\tiny 5};
\draw(7.5,6) node[above] {\tiny 4};
\draw(-2,2) node[left] {\tiny $v_1$};
\draw(-2,4) node[left] {\tiny $v_1$};
\draw(2,2) node[above right] {\tiny $v_1$};
\draw(2,4) node[below right] {\tiny $v_1$};
\draw(0,2) node[right] {\tiny $v_2$};
\draw(0,4) node[right] {\tiny $v_2$};
\draw(8,4) node[right] {\tiny $v_2$};
\draw(8,2) node[right] {\tiny $v_2$};
\end{tikzpicture}
\end{minipage}
 \caption{Cylinder Diagrams 3B) (left) and 3C) (right)}
 \label{Config3BCCylDiagsFig}
\end{figure}

\begin{remark}
For all three cylinder diagrams in this configuration in the principal stratum, the arguments of Lemmas \ref{CylDiags3BC} and \ref{CylDiags3A} work even if saddle connection $6$ has zero length, i.e. there is a double zero in place of saddle connection $6$.
\end{remark}

\begin{lemma}
\label{CylDiags3BC}
Let $\mathcal{M}$ be a rank one affine manifold with non-trivial REL and $2$-dimensional Forni subspace.  If $M \in \mathcal{M}$ satisfies Cylinder Diagrams 3B) or 3C), then either there exists $M' \in \cM$ satisfying Configuration 4) or there exists $\cM'$ in a lower dimensional stratum in genus three such that $\cM'$ has a $2$-dimensional Forni subspace.
\end{lemma}

\begin{remark}
In Cylinder Diagram 3B), the argument goes through even if saddle connection $2$ has zero length.
\end{remark}

\begin{proof}
If $\text{Twist}(M, \mathcal{M}) \not= \text{Pres}(M, \mathcal{M})$, then by Lemma \ref{WrightLem86}, there exists a translation surface $M' \in \cM$ with more cylinders.  However, $\cM$ has a $2$-dimensional Forni subspace, and the only cylinder configuration with more than three cylinders is Configuration 4) by Lemma \ref{Gen3TopConfigs}.

Next assume that $\text{Twist}(M, \mathcal{M}) = \text{Pres}(M, \mathcal{M})$ so that Theorem \ref{RELDefs} implies the existence of a REL stretch and a REL twist.  In Cylinder Diagrams 3B) and 3C), we argue that it is possible to collapse saddle connections with a REL stretch without degenerating the surface so that we can conclude by Lemma \ref{DetLocusClosed}.

Note that in Cylinder Diagrams 3B) and 3C), the cylinder $C_1$ contains a non-trivial element of absolute homology with non-zero period traversing the height of $C_1$ by connecting the two copies of saddle connection $1$.  Since the core curve of any cylinder also represents a non-trivial element of absolute homology with non-zero period, the cylinder $C_1$ is completely fixed by any REL stretch.  Since $C_2$ and $C_3$ are homologous, the observation about $C_1$ implies that the sum of the heights $h_2 + h_3$ represent the imaginary part of an absolute period.

In both cylinder diagrams perform a REL twist so that $C_3$ admits a vertical saddle connection, and REL stretch $C_3$ until it collapses.  Clearly $C_2$ persists under this degeneration because its height converges to the quantity $h_2 + h_3 > 0$.  Let $M'$ be the resulting degenerate surface with two horizontal cylinders.  We claim that no curve pinches in either cylinder diagram under this degeneration.  We follow the usual argument and show that no union of saddle connections forming a closed curve degenerates to a point.  We consider each diagram separately.

\

\noindent \textit{Cylinder Diagram 3B):}  Observe that all horizontal saddle connections remain fixed under this degeneration.  Furthermore, all saddle connections traversing the height of $C_1$ or $C_2$ have length bounded away from zero because neither of the heights of $C_1$ nor $C_2$ converge to zero.  Hence, only a saddle connection entirely contained in $C_3$ can converge to zero length.  Moreover, since all horizontal (real) components of periods are fixed under a REL stretch, only a vertical saddle connection contained in $C_3$ can converge to zero length.  By construction, exactly one such saddle connection does.  We leave it to the reader to check that the set of zeros contained in the top of $C_3$ is disjoint from the set of zeros contained in the bottom of $C_3$ from which the claim follows that the degeneration yields a translation surface of genus three in a lower dimensional stratum.

\

\noindent \textit{Cylinder Diagram 3C):} This case is slightly more complicated due to the fact that there can be up to two vertical saddle connections that collapse after collapsing the cylinder $C_3$.  Observe that the zeros $v_1$ and $v_2$ in Cylinder Diagram 3C) must always be simple and distinct.  (If they were not distinct, saddle connection 1 would have zero length, and this case will be addressed via Cylinder Diagram 3A) below.)  Following the same argument above, it suffices to analyze the vertical saddle connections that are entirely contained in $C_3$.

Observe that the set of zeros on the top of $C_3$, $\{v_1, v_2\}$, is disjoint from the set of zeros on the bottom of $C_3$.  The zeros on the bottom of $C_3$ either consists of two simple zeros, or one double zero (when saddle connection 6 has zero length).  Let $v_3$ be a zero on the bottom of $C_3$.  Since there is only one copy of each of $v_1$ and $v_2$ on the top of $C_3$, there can only be at most one vertical saddle connection emanating from each of them that is contained in $C_3$.  Hence, even if there are two vertical saddle connections from $v_1$ to $v_3$ and $v_2$ to $v_3$, the union of these two saddle connections cannot represent a closed curve.  The conclusion is even clearer in the case of a single vertical saddle connection, or in the case that one of the copies of $v_3$ is replaced by a fourth simple zero $v_4$.  Thus, the proof is complete.
\end{proof}

\begin{lemma}
\label{CylDiags3A}
Let $\mathcal{M}$ be a rank one affine manifold with non-trivial REL and $2$-dimensional Forni subspace.  If $M \in \mathcal{M}$ satisfies Cylinder Diagram 3A), then there exists $M' \in \cM$ satisfying Configuration 4).
\end{lemma}

\begin{remark}
In Cylinder Diagram 3A), the following proof works just as well if saddle connection $2$ has zero length, and the details will be given in a footnote at the end of the proof.
\end{remark}

\begin{proof}
If $\text{Twist}(M, \mathcal{M}) \not= \text{Pres}(M, \mathcal{M})$, then by Lemma \ref{WrightLem86}, there exists a translation surface $M' \in \cM$ with more cylinders.  However, $\cM$ has a $2$-dimensional Forni subspace, and the only cylinder configuration with more than three cylinders is Configuration 4) by Lemma \ref{Gen3TopConfigs}.

Assume that $\text{Twist}(M, \mathcal{M}) = \text{Pres}(M, \mathcal{M})$.  Following the proof of Lemma \ref{CylDiag2}, we find a basis for homology that allows us to apply Corollary \ref{HomBasisForniContraCor}.  One difference here that was not used in Lemma \ref{CylDiag2} is that it was possible in Configuration 2) to realize all of the desired properties of the basis on a single translation surface.  However, in this lemma, we will need the full power of Corollary \ref{HomBasisForniContraCor} to argue that it suffices to find a translation surface in the REL leaf for each absolute homology curve such that the curve can be realized as the core curve of a cylinder.

Through a combination of \splin ~and REL deformations, we can arrange the cylinder diagram as in Figure \ref{Config3ACylDiagsFig}.  Let $C_i$ have height $h_i$.  Observe the absolute homology cycle $b_1$ in Figure \ref{Config3ACylDiagsFig}, which has imaginary part of its (absolute) period equal to $h_1 + h_2 + h_3$.  Furthermore, the existence of the absolute homology cycle $b_2$, which has imaginary part of its (absolute) period equal to $h_2 + h_3$ shows that $h_1$ is fixed under any REL stretch.

Therefore, we can REL stretch so that $C_3$ has height $\delta > 0$, where $\delta$ is much smaller than the shortest horizontal saddle connection on $C_3$, and $C_1$ remains fixed.  All of the claims below can be verified by letting $\delta$ go to zero, finding the trajectory in $C_2 \cup C_3$ (which becomes a single cylinder when $\delta = 0$) and then observing that after a possible small perturbation to the trajectory, it will persist for sufficiently small $\delta > 0$.  Unfortunately, it is not sufficient to collapse $C_3$ in this argument because it may result in a surface of lower genus.  Hence, the arguments below are made near the boundary of $\cM$, but \emph{not} at the boundary of $\cM$.

We follow the notation of Corollary \ref{HomBasisForniContraCor}, and let $M_{\gamma}$ denote a translation surface for which $\gamma$ is the core curve of a cylinder.  First, REL twist the cylinders $C_2$ and $C_3$ to produce the surface $M_{b_1}$ so that a closed curve corresponding to $b_1$ can pass through saddle connection $3$ and travel from $0$ to itself.  Secondly, the same can be done with saddle connection $5$ so that $b_2$ can pass through and connect saddle connection $1$ to itself, which yields $M_{b_2}$.  Finally, we can realize $M_{b_3}$ by performing a REL twist so that saddle connection $4$ lies directly above saddle connection $2$ in $C_3$.\footnote{The only significant difference in this proof when $2$ has length zero is that $b_3$ is forced to pass through saddle connection $1$ and intersect $b_2$.  In spite of this intersection, $b_2$ and $b_3$ represent linearly independent absolute homology curves because $I(b_2, a_3) = 0$ and $I(b_3, a_3) \not= 0$.}  Since the height $\delta$ of $C_3$ is negligible, connecting $2$ to itself by passing through $4$ is easy.  Again, this can be seen by letting $\delta = 0$, connecting subset of $4$ corresponding to $2$ to the copy of $2$ on the top of $C_2$, and noting that this cylinder will persist for sufficiently small $\delta > 0$.  Applying Corollary \ref{HomBasisForniContraCor}, yields the desired contradiction and proves $\text{Twist}(M, \mathcal{M}) \not = \text{Pres}(M, \mathcal{M})$.
\end{proof}

\subsection{Configuration 4)}

The final case to consider in this section is Configuration 4).  Though this case admits only two cylinder diagrams in the principal stratum and two cylinder diagrams in $\cH(2,1,1)$, a new problem arises from the more complicated admissible relations among the core curves of the cylinders in one of the cylinder diagrams.  Our assumptions only tell us that REL is positive dimensional, so we cannot guarantee more than one dimension.  Several homological relations are possible and we do not know which ones can be varied by moving in REL.  Therefore, we have to exclude all possibilities.

\begin{remark}
The reader may check that if a cylinder diagram in genus three satisfies Configuration 4), then it must lie in
$$\cH(2,1,1) \cup \cH(1,1,1,1).$$
Throughout this section, we focus on the principal stratum.  Every argument below is independent of the length of saddle connection $6$, which we permit to have zero length.
\end{remark}

\begin{figure}[ht]
\begin{minipage}[b]{0.24\linewidth}
\centering
\begin{tikzpicture}[scale=0.50]
\draw (0,0)--(0,4)--(2,4)--(2,2)--(2.5,4)--(5.5,4)--(5,2)--(5,0)--(8,0)--(8,-2)--(3,-2)--(3,0)--cycle;
\foreach \x in {(0,0),(0,2),(2,2),(0,4),(2,4),(2.5,4),(5.5,4),(5,2),(5,0),(8,0),(8,-2),(5,-2),(3,-2),(3,0)} \draw \x circle (1pt);
\foreach \x in {(1,0),(2,0),(6,0),(7,0)} \draw \x circle (1pt);
\draw(2.5,1) node {\tiny $C_3$};
\draw(5.5,-1) node {\tiny $C_4$};
\draw(1,3) node {\tiny $C_1$};
\draw(4,3) node {\tiny $C_2$};
\draw(1,4) node[above] {\tiny 0};
\draw(4,4) node[above] {\tiny 1};
\draw(4,-2) node[below] {\tiny 0};
\draw(6.5,-2) node[below] {\tiny 1};
\draw(.5,0) node[below] {\tiny 3};
\draw(1.5,0) node[below] {\tiny 4};
\draw(2.5,0) node[below] {\tiny 5};
\draw(5.5,0) node[above] {\tiny 5};
\draw(6.5,0) node[above] {\tiny 4};
\draw(7.5,0) node[above] {\tiny 3};
\end{tikzpicture}
\end{minipage}
\begin{minipage}[t]{0.24\linewidth}
\centering
\begin{tikzpicture}[scale=0.30]
\end{tikzpicture}
\end{minipage}
\begin{minipage}[b]{0.24\linewidth}
\centering
\begin{tikzpicture}[scale=0.50]
\draw (0,0)--(0,2)--(-0.5,4)--(1.5,4)--(2,2)--(2,4)--(10,4)--(10,-2)--(2,-2)--(2,0)--cycle;
\foreach \x in {(0,2),(2,2),(10,2),(0,0),(2,0),(10,0),(-0.5,4),(1.5,4),(2,4),(4,4),(6,4),(8,4),(10,4),(10,-2),(8,-2),(6,-2),(4,-2),(2,-2)} \draw \x circle (1pt);
\draw(5,1) node {\tiny $C_2$};
\draw(6,-1) node {\tiny $C_4$};
\draw(0.5,3) node {\tiny $C_1$};
\draw(6,3) node {\tiny $C_3$};
\draw(0.5,4) node[above] {\tiny 0};
\draw(3,4) node[above] {\tiny 6};
\draw(5,4) node[above] {\tiny 5};
\draw(7,4) node[above] {\tiny 4};
\draw(9,4) node[above] {\tiny 3};
\draw(1,0) node[below] {\tiny 0};
\draw(3,-2) node[below] {\tiny 3};
\draw(5,-2) node[below] {\tiny 4};
\draw(7,-2) node[below] {\tiny 5};
\draw(9,-2) node[below] {\tiny 6};
\end{tikzpicture}
\end{minipage}
 \caption{Cylinder Diagrams 4A) (left) and 4B) (right)}
 \label{Config4CylDiagFig}
\end{figure}
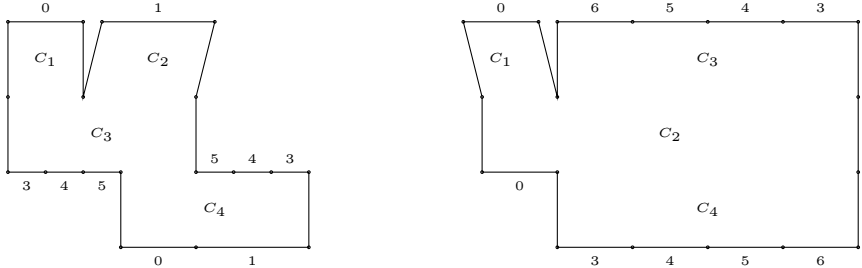

\begin{lemma}
\label{CylDiags4}
There are two $4$-cylinder diagrams satisfying Configuration 4) up to reflection, and they are depicted in Figure \ref{Config4CylDiagFig}.
\end{lemma}

\begin{proof}
As noted above, there is exactly one way to identify two cylinders each having two simple zeros on their boundary so that the resulting identification is a torus.  This is because there is one $1$-cylinder diagram in $\mathcal{H}(1,1)$.  Furthermore, there is exactly one way to identify three cylinders with a single simple zero lying between them, and this corresponds to the unique $3$-cylinder diagram in $\mathcal{H}(1,1)$.  There are two choices for cylinders in the $3$-cylinder diagram in $\mathcal{H}(1,1)$ that can be cut and replaced by two homologous cylinders with a torus lying in between.  The two choices correspond to the two cylinder diagrams in Figure \ref{Config4CylDiagFig}.  (We say that there are two choices instead of three because the simple cylinders in the $3$-cylinder diagram in $\mathcal{H}(1,1)$ are indistinguishable for our purposes here.)
\end{proof}

\begin{lemma}
\label{Config4TwEqPr}
If $\mathcal{M}$ is an orbit closure in genus three with $2$-dimensional Forni subspace and $M \in \cM$ satisfies Configuration 4), then
$$\text{Twist}(M, \mathcal{M}) = \text{Pres}(M, \mathcal{M}).$$
\end{lemma}

\begin{proof}
In Configuration 4), we have achieved the maximum number of cylinders possible in an affine manifold in genus three with non-trivial Forni subspace by Lemma \ref{Gen3TopConfigs}, thus the lemma follows.
\end{proof}

\

\noindent \textbf{Cylinder Diagram 4B)}:  This case is simpler because the desired absolute homology basis can be seen directly without using REL deformations.

\begin{figure}[ht]
\centering
\begin{tikzpicture}[scale=0.50]
\draw (0,0)--(0,5)--(1.9,5)--(2,2)--(2.1,3)--(10,3)--(10,-2)--(2,-2)--(2,0)--cycle;
\foreach \x in {(1,5),(10,2),(0,2),(2,2),(0,0),(2,0),(10,0)} \draw \x circle (1pt);
\foreach \x in {(2.1,3),(4,3),(7,3),(9,3),(10,3),(8,-2),(6,-2),(5,-2),(3,-2)} \draw \x circle (1pt);
\draw(1.5,5) node[above] {\tiny 0'};
\draw(0.5,5) node[above] {\tiny 0};
\draw(1,0) node[below] {\tiny 0};
\draw(3,3) node[above] {\tiny 6};
\draw(5.5,3) node[above] {\tiny 5};
\draw(8,3) node[above] {\tiny 4};
\draw(9.5,3) node[above] {\tiny 3};
\draw(9,-2) node[below] {\tiny 5};
\draw(7,-2) node[below] {\tiny 4};
\draw(5.5,-2) node[below] {\tiny 3};
\draw(4,-2) node[below] {\tiny 6};
\draw(2.5,-2) node[below] {\tiny 5'};

\draw [->, cyan] (.1,0) to [out=90,in=-90] (1.1,5);
\node [above] at (1,0) {$b_1$};
\draw [->, cyan] (5.5,-2) to [out=90,in=-90] (9.5,3);
\node [above right] at (9,1) {$b_2$};
\draw [->, cyan] (7,-2) to [out=90,in=-90] (8,3);
\node [above left] at (8,1) {$b_3$};

\draw [->, red] (5.9,-2) to [out=90,in=90] (8.1,-2);
\draw [<-, red] (9.9,3) to [out=-90,in=180] (10,2.5);
\draw [<-, red] (2.05,2.5) to [out=0,in=-90] (4.1,3);
\node [above right] at (8.1,-2) {$a_3$};

\draw [->, red] (0,3.5) -- (1.95,3.5);
\node [left] at (0,3.5) {$a_1$};
\draw [->, red] (2,-1) -- (10,-1);
\node [right] at (10,-1) {$a_2$};

\end{tikzpicture}
 \caption{Choice of absolute homology basis for Cylinder Diagram 4B)}
 \label{Config4BHomBasisFig}
\end{figure}
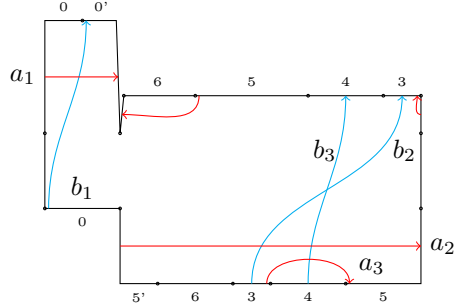

\begin{lemma}
\label{CylDiags4B}
If $\mathcal{M}$ is a rank one orbit closure with $2$-dimensional Forni subspace and non-trivial REL, then it does not contain a translation surface with Cylinder Diagram 4B).
\end{lemma}

\begin{proof}
We find a basis for homology that allows us to apply Corollary \ref{HomBasisForniContraCor}.  We claim that the cylinder diagram in this case can be depicted as in Figure \ref{Config4BHomBasisFig}.  To see this, act on the surface by the horocycle flow so that the saddle connection on the bottom of $C_1$ lies over the copy of $0$ on the bottom of $C_2$, and the saddle connection on the bottom of $C_3$ lies over the saddle connection on the top of $C_4$.  No assumptions are made about the saddle connections in the top and bottom of $C_3$ and $C_4$, respectively.  Consider the $a_i$ as depicted in Figure \ref{Config4BHomBasisFig}.  The existence of the closed trajectory forming $b_1$ is immediate.  Furthermore, every saddle connection $3, 4, 5, 6$ on the top of $C_3$ is visible from its copy on the bottom of $C_4$.  This implies that $b_2$ and $b_3$ can be chosen so that they are core curves of cylinders as well.  Having satisfied all of the assumptions of Corollary \ref{HomBasisForniContraCor}, we can apply it to complete this proof.
\end{proof}

\

\noindent \textbf{Cylinder Diagram 4A)}:  We consider the possible deformations that can occur when we perform a REL stretch.  Let $h_i$ denote the height of the cylinder $C_i$.  Observe that there exist curves representing absolute homology cycles $b_1$, $b_2$ passing from the bottom of $C_4$, through $C_3$, and into $C_1$ and $C_2$, respectively, before closing.  The imaginary part of the (absolute) period of $b_i$ is given by $h_i + h_3 + h_4$, for $i = 1,2$.  Since we will consider REL deformations, this quantity must remain fixed.  Let $h_i + h_3 + h_4 = H_i$, for $H_1, H_2 > 0$.  In particular, we have $h_1 - h_2 = H_1 - H_2$, which implies that $h_2$ is completely determined by $h_1$, $h_3$, and $h_4$ along any REL deformation.

We focus on the equation
$$h_1 + h_3 + h_4 = H_1.$$
The complicating issue for this cylinder diagram is the fact that $\cM$ may be $3$-dimensional, i.e. have one relative dimension.  When $\cM$ is $3$-dimensional, there is an additional unknown equation relating the quantities $h_1, h_3, h_4$.  By the definition of an affine manifold, such an equation has the form
$$\alpha_1 h_1 + \alpha_3 h_3 + \alpha_4 h_4 = H$$
for some $\alpha_1, \alpha_3, \alpha_4, H \in \bR$.  The strategy of this section is to exhaustively study every possible equation.  The three cases we consider are as follows.  Either
\begin{itemize}
\item (L1) the heights $h_1$ and $h_2$ are fixed under the REL stretch, i.e. there exists $H > 0$ such that
$$h_1 = H,$$
\item (L2) the height $h_3$ or $h_4$ (but not both) is fixed under a REL stretch, i.e. there exists $H > 0$ such that
$$h_3 = H \text{ or } h_4 = H, \text{ or}$$
\item (L3) no height of any cylinder is fixed under a REL stretch, i.e. $h_1$ and $h_3$ can be expressed as non-constant linear functions depending on the free variable $h_4$.
\end{itemize}

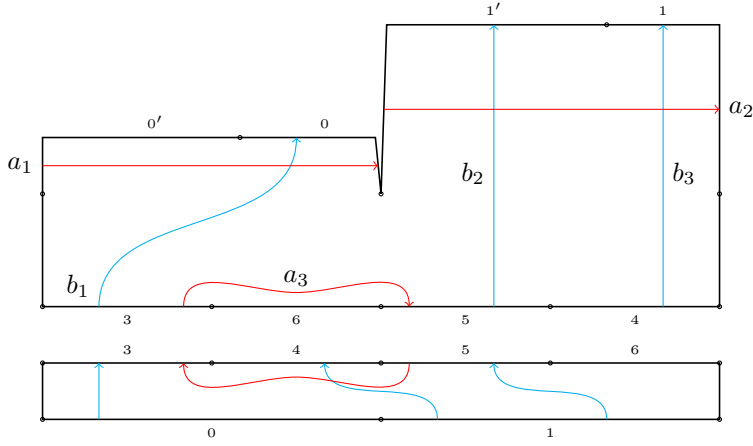
\begin{figure}
\centering
\begin{tikzpicture}[scale=0.75]
\draw [semithick] (0,1)--(0,4)--(5.9,4)--(6,3)--(6.1,6)--(12,6)--(12,1)--cycle;
\draw [semithick] (0,-1)--(0,0)--(12,0)--(12,-1)--cycle;
\foreach \x in {(3.5,4),(0,3),(6,3),(10,6),(12,3),(0,0),(3,0),(6,0),(9,0),(12,0),(0,1),(3,1),(6,1),(9,1),(12,1), (0,-1),(6,-1),(12,-1)} \draw \x circle (1pt);
\draw(5,4) node[above] {\tiny $0$};
\draw(2,4) node[above] {\tiny $0'$};
\draw(11,6) node[above] {\tiny $1$};
\draw(8,6) node[above] {\tiny $1'$};
\draw(1.5,1) node[below] {\tiny $3$};
\draw(4.5,1) node[below] {\tiny $6$};
\draw(7.5,1) node[below] {\tiny $5$};
\draw(10.5,1) node[below] {\tiny $4$};
\draw(1.5,0) node[above] {\tiny $3$};
\draw(4.5,0) node[above] {\tiny $4$};
\draw(7.5,0) node[above] {\tiny $5$};
\draw(10.5,0) node[above] {\tiny $6$};
\draw(3,-1) node[below] {\tiny $0$};
\draw(9,-1) node[below] {\tiny $1$};

\draw [->, cyan] (1,-1) -- (1,0);
\draw [->, cyan] (1,1) to [out=90,in=-90] (4.5,4);
\node [above left] at (1,1) {$b_1$};
\draw [->, cyan] (10,-1) to [out=90,in=-90] (8,0);
\draw [->, cyan] (8,1) -- (8,6);
\node [above left] at (8,3) {$b_2$};
\draw [->, cyan] (7,-1) to [out=90,in=-90] (5,0);
\draw [->, cyan] (11,1) to [out=90,in=-90] (11,6);
\node [above right] at (11,3) {$b_3$};

\draw [->, red] (2.5,1) to [out=90,in=180] (4.5,1.25) to [out=0,in=90] (6.5,1);
\draw [<-, red] (2.5,0) to [out=-90,in=180] (4.5,-.25) to [out=0,in=-90] (6.5,0);
\node [above] at (4.5,1.2) {$a_3$};
\draw [->, red] (0,3.5) -- (5.95,3.5);
\node [left] at (0,3.5) {$a_1$};
\draw [->, red] (6.05,4.5) -- (12,4.5);
\node [right] at (12,4.5) {$a_2$};
\end{tikzpicture}
 \caption{Cylinder Diagram 4A) Case (L1) with a choice of homology basis}
 \label{Config4CylDiagL1Fig}
\end{figure}

\begin{lemma}
\label{CylDiags4L1}
If $\mathcal{M}$ is a rank one orbit closure with $2$-dimensional Forni subspace and non-trivial REL, then it does not contain a translation surface with Cylinder Diagram 4A) Case (L1).
\end{lemma}

\begin{proof}
This proof will follow similarly to that of Lemma \ref{CylDiags3A} because the main assumptions are similar.  As in the previous lemmas, we find a basis for homology that allows us to apply Corollary \ref{HomBasisForniContraCor}.  By Lemma \ref{Config4TwEqPr} and Theorem \ref{RELDefs}, there exists a non-trivial REL deformation.

We claim that the cylinder diagram in Case (L1) can be depicted as in Figure \ref{Config4CylDiagL1Fig} and that each of the homology curves $b_i$ can be realized as core curves of cylinders for all $i$.  The Case (L1) assumption states that the heights $h_1$ and $h_2$ are fixed and the quantity $h_3 + h_4$ is constant.  Thus, we may treat $h_4$ as a free variable and let $h_4 = \delta$, for some $\delta > 0$ much smaller than the length of the shortest horizontal saddle connection.  As in the proof of Lemma \ref{CylDiags3A}, all of the claims below can be verified by letting $\delta$ go to zero, finding the trajectory in $C_1 \cup C_2 \cup C_3$, and then observing that after returning to the interior of $\cM$ in a neighborhood of the boundary, the trajectory determining a cylinder will persist for sufficiently small $\delta > 0$.

Following our usual notation, let $M_{\gamma}$ denote a translation surface that realizes the curve $\gamma$ as a core curve of a cylinder.  We claim that each $b_i$ in Figure \ref{Config4CylDiagL1Fig} can be realized as the core curve of a cylinder.  Throughout, let $M$ be a horizontally periodic translation surface satisfying Cylinder Diagram 4A) Case (L1).  Once each of the surfaces below have been shown to exist in $\cM$, we apply Corollary \ref{HomBasisForniContraCor} to complete the proof.

\

\noindent $M_{b_1}$:  Act by the horocycle flow and a REL twist so that saddle connection $3$ lies over saddle connection $0$ in $C_4$ and the saddle connection in the bottom of $C_1$ lies over saddle connection $3$ in $C_3$.  For $\delta$ sufficiently small, $b_1$ can be realized as a closed trajectory determining a cylinder.

\

\noindent $M_{b_2}$:  Act by the horocycle flow and a REL twist so that saddle connection $5$ lies over saddle connection $1$ in $C_4$ and the saddle connection in the bottom of $C_2$ lies over saddle connection $5$ in $C_3$.  For $\delta$ sufficiently small, $b_2$ can be realized as a closed trajectory determining a cylinder.

\

\noindent $M_{b_3}$:  Act by the horocycle flow and a REL twist so that saddle connection $4$ lies over saddle connection $1$ in $C_4$ and the saddle connection in the bottom of $C_2$ lies over saddle connection $4$ in $C_3$.  For $\delta$ sufficiently small, $b_3$ can be realized as a closed trajectory determining a cylinder.
\end{proof}

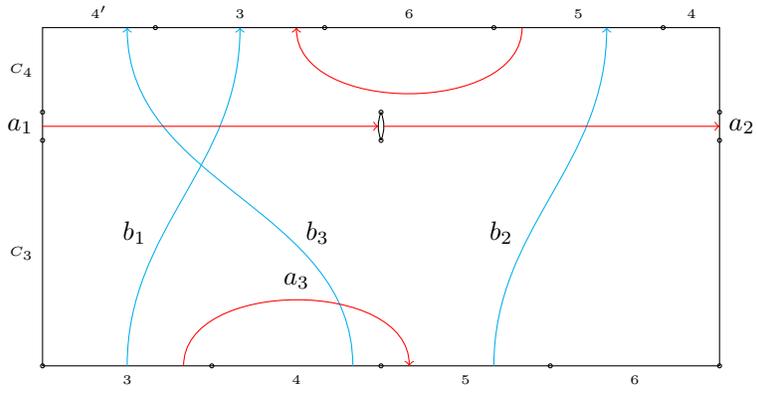
\begin{figure}
\centering
\begin{tikzpicture}[scale=0.75]
\draw (0,0)--(0,6)--(12,6)--(12,0)--cycle;
\draw (6,4) to [out=110,in=-110] (6,4.5);
\draw (6,4) to [out=70,in=-70] (6,4.5);

\foreach \x in {(0,0),(3,0),(6,0),(9,0),(12,0),(0,4),(0,4.5),(6,4),(6,4.5),(12,4),(12,4.5),(2,6),(5,6),(8,6),(11,6)} \draw \x circle (1pt);
\draw(0,2) node[left] {\tiny $C_3$};
\draw(0,5.25) node[left] {\tiny $C_4$};

\draw(1.5,0) node[below] {\tiny $3$};
\draw(4.5,0) node[below] {\tiny $4$};
\draw(7.5,0) node[below] {\tiny $5$};
\draw(10.5,0) node[below] {\tiny $6$};
\draw(3.5,6) node[above] {\tiny $3$};
\draw(6.5,6) node[above] {\tiny $6$};
\draw(9.5,6) node[above] {\tiny $5$};
\draw(11.5,6) node[above] {\tiny $4$};
\draw(1,6) node[above] {\tiny $4'$};

\draw [->, cyan] (1.5,0) to [out=90,in=-90] (3.5,6);
\node [above left] at (2,2) {$b_1$};
\draw [->, cyan] (8,0) to [out=90,in=-90] (10,6);
\node [above left] at (8.5,2) {$b_2$};
\draw [->, cyan] (5.5,0) to [out=90,in=-90] (1.5,6);
\node [above right] at (4.5,2) {$b_3$};

\draw [->, red] (2.5,0) to [out=90,in=90] (6.5,0);
\draw [<-, red] (4.5,6) to [out=-90,in=-90] (8.5,6);
\node [above] at (4.5,1.2) {$a_3$};
\draw [->, red] (0,4.25) -- (5.95,4.25);
\node [left] at (0,4.25) {$a_1$};
\draw [->, red] (6.05,4.25) -- (12,4.25);
\node [right] at (12,4.25) {$a_2$};

\end{tikzpicture}
 \caption{Cylinder Diagram 4A) Case (L2) with a choice of homology basis}
 \label{Config4CylDiagL2Fig1}
\end{figure}

\begin{figure}
\centering
\begin{tikzpicture}[scale=0.75]
\draw (0,0)--(0,4)--(8,4)--(8,0)--cycle;
\foreach \x in {(0,0),(3.5,0),(5,0),(7,0),(8,0),(0,4),(3.5,4),(4.5,4),(6.5,4),(8,4),(3.75,3),(7.75,3)} \draw \x circle (1pt);
\draw (0,3)--(3.75,3);
\draw (7.75,3)--(8,3);

\draw(1.75,0) node[below] {\tiny $3$};
\draw(4.25,0) node[below] {\tiny $4$};
\draw(6,0) node[below] {\tiny $5$};
\draw(7.5,0) node[below] {\tiny $6$};
\draw(1.75,4) node[above] {\tiny $3$};
\draw(4,4) node[above] {\tiny $6$};
\draw(5.5,4) node[above] {\tiny $5$};
\draw(7.25,4) node[above] {\tiny $4$};

\draw [->, cyan] (2,0)--(2,4);
\node [above left] at (2,1) {$b_1$};
\draw [->, cyan] (6,0) to [out=90,in=-90] (5.5,4);
\node [above right] at (5.9,1) {$b_2$};

\end{tikzpicture}
 \caption{Cylinder Diagram 4A) Case (L2) Proof of Lemma \ref{CylDiags4L2Lem2}}
 \label{Config4CylDiagL2Fig2}
\end{figure}
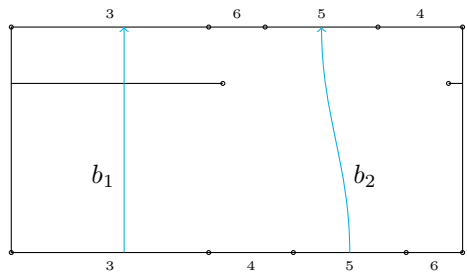

Given a simple cylinder $C$ in any translation surface, $C$ can be naturally associated to a point in $\Omega \cM_{1,1}$, the moduli space of once punctured tori.  Consider the cylinder $C$ and remove the zero or zeros in its boundary and mark the resulting points.  After identifying opposite sides we get a torus with a single marked point.  Two simple cylinders are \emph{isometric} if they have equal area and correspond to the same point in $\Omega \cM_{1,1}$.

\begin{lemma}
\label{CylDiags4L2Lem1}
Let $\mathcal{M}$ be a rank one orbit closure with $2$-dimensional Forni subspace and non-trivial REL.  If $\cM$ contains a translation surface satisfying Cylinder Diagram 4A) Case (L2), then either
\begin{itemize}
\item There exists $\cM'$ (in the boundary of $\cM$) in a lower dimensional stratum in genus three with a $2$-dimensional Forni subspace, or
\item $C_1$ and $C_2$ are isometric.
\end{itemize}
\end{lemma}

\begin{proof}
By Lemma \ref{Config4TwEqPr} and Theorem \ref{RELDefs}, there exists a non-trivial REL deformation.  Assume that $C_1$ and $C_2$ are not isometric, otherwise we are done.  Without loss of generality, let $h_4 = H$, for some fixed $H > 0$ and let $h_1 \leq h_2$.  Deform the surface by the horocycle flow so that $C_1$ has a vertical saddle connection.  If $h_1 < h_2$, then as observed above, if we let $h_1$ tend to zero by a REL stretch, we get $h_2 = H_2 - H_1 > 0$ so this collapses two distinct simple zeros by inspection and we conclude by Lemma \ref{DetLocusClosed}.  Hence, it suffices to consider $h_1 = h_2$.  Furthermore, the same argument applies if the two cylinders do not simultaneously contain vertical saddle connections, i.e. have equal ``twist.''  Otherwise, a REL stretch could collapse two zeros (exactly one saddle connection) without collapsing a curve.  Finally, if the circumferences of $C_1$ and $C_2$ are not equal, act by a horocycle element on $M$ so that we perform a Dehn twist on the smaller cylinder.  The larger cylinder will not have undergone a Dehn twist and therefore, it can no longer admit a vertical saddle connection.  Hence, if $C_1$ and $C_2$ are not isometric, we can always degenerate to a lower dimensional stratum.
\end{proof}

The following lemma proves that in fact $C_1$ and $C_2$ can never be isometric.

\begin{lemma}
\label{CylDiags4L2Lem2}
Let $\mathcal{M}$ be a rank one orbit closure with $2$-dimensional Forni subspace and non-trivial REL.  If $\cM$ contains a translation surface satisfying Cylinder Diagram 4A) Case (L2), then $C_1$ and $C_2$ are not isometric.
\end{lemma}

\begin{proof}
By contradiction, assume that $C_1$ and $C_2$ are isometric.  As in previous lemmas, we find a basis for homology that allows us to apply Corollary \ref{HomBasisForniContraCor}.  We claim that if $C_1$ and $C_2$ are isometric, then the cylinder diagram in Case (L2) can be depicted as in Figure \ref{Config4CylDiagL2Fig1}.  By Lemma \ref{Config4TwEqPr} and Theorem \ref{RELDefs}, there exists a non-trivial REL deformation.  Since $C_1$ and $C_2$ are isometric, we can apply a combination of the horocycle flow and a REL twist so that each of $C_1$, $C_2$, and $C_3$ admit a vertical saddle connection.  The vertical saddle connection can be seen in Figure \ref{Config4CylDiagL2Fig1} as the sides of the cylinders $C_1$, $C_2$, $C_3$.  Next, perform a REL stretch to shrink the heights of $C_1$ and $C_2$ to any arbitrarily small quantity, which can be done by the Case (L2) assumption.

We consider a specific boundary translation surface of $\cM$ where many arguments clearly apply and then, following the theme of the proof of many lemmas in this section, argue that the desired property that exists on the boundary persists in a neighborhood in the interior of $\cM$.  In Figure \ref{Config4CylDiagL2Fig1}, continue the REL stretch so that the cylinders $C_1$ and $C_2$ collapse to marked horizontal line segments.  In Figure \ref{Config4CylDiagL2Fig2}, we mark the degeneration of $C_1$ with a horizontal line segment, and the degeneration of $C_2$ is left blank.  Observe that after collapsing $C_1$ and $C_2$, $C_3$ and $C_4$ can be viewed as a single cylinder.  After cutting and regluing if necessary and twisting by the horocycle flow, the resulting degenerate surface appears as in Figure \ref{Config4CylDiagL2Fig2}.  Without loss of generality, let $3$ be the longest of the four saddle connections $\{3, 4, 5, 6\}$, and let it be in the location depicted in Figure \ref{Config4CylDiagL2Fig2}.

Next we define the homology basis $\{a_1, a_2, a_3, b_1, b_2, b_3\}$ depicted in Figure \ref{Config4CylDiagL2Fig1}.  The curves $a_1$ and $a_2$ are the core curves of the cylinders $C_1$ and $C_2$, respectively.  The absolute homology curve $a_3$ is exactly as depicted in Figure \ref{Config4CylDiagL2Fig1}, where it is homotopic to a union of saddle connections $4$ and $6$.  The curve $b_1$ satisfies the intersection property $I(a_1, b_1) \not= 0$ and $I(a_i, b_1) = 0$, for $i = 2,3$.  The curve $b_2$ satisfies the intersection property $I(a_2, b_2) \not= 0$ and $I(a_i, b_2) = 0$, for $i = 1,3$.  The curve $b_3$ satisfies the intersection property $I(a_3, b_3) \not= 0$.  The fact that $\{a_1, a_2, a_3\}$ form a Lagrangian subspace of homology combined with the intersection assumptions on the $b_j$, imply that the collection of all $a_i$ and $b_j$ are linearly independent, whence they are a basis for absolute homology.  We make no assumptions about which saddle connections $b_1$, $b_2$, and $b_3$ cross as long as the intersection relations are satisfied.

For convenience assume that the cylinder in Figure \ref{Config4CylDiagL2Fig2} has unit circumference.  Under this normalization, $C_1$ and $C_2$ each have circumference $1/2$.  Thus, the marked horizontal segment in Figure \ref{Config4CylDiagL2Fig2} also has length $1/2$.  Let $\ell(\cdot)$ denote the length of a saddle connection.

In order to apply Corollary \ref{HomBasisForniContraCor} to conclude, we show that each curve $b_i$ can be realized as the core curve of a cylinder.  It is clear that $b_3$ can be realized as the core curve of a cylinder from saddle connection $4$ to itself because no assumptions were made about the intersection of $b_3$ with $a_1$ or $a_2$.  After renumbering $C_1$ and $C_2$, thus $a_1$ and $a_2$, it suffices to assume that $b_1$ can be realized as the core curve of a vertical cylinder from $3$ to itself in Figure \ref{Config4CylDiagL2Fig2}.  Depending on the length of saddle connection $3$ and the location of the horizontal segment in Figure \ref{Config4CylDiagL2Fig2}, there may be more than one vertical cylinder from saddle connection $3$ to itself.  In fact, if there is more than one cylinder, then a second cylinder would pass through the blank space in the complement of the horizontal segment, which corresponds to $a_2$ in Figure \ref{Config4CylDiagL2Fig1}.  This would allow us to realize $b_2$ as the core curve of a vertical cylinder from $3$ to itself and we would be done.  Since the horizontal segment always has length $1/2$, if $\ell(3) > 1/2$, then we are done because such a second vertical cylinder from $3$ to itself realizing $b_2$ would exist.  Therefore, it suffices to assume that $\ell(3) \leq 1/2$ and that there is a unique vertical cylinder from $3$ to itself.  Since $3$ was assumed to be the longest saddle connection and $\sum_{j=3}^6 \ell(j) = 1$, we have $1/4 \leq \ell(3) \leq 1/2$.

Having realized $b_1$ and $b_3$ as core curves of cylinders, it suffices to realize $b_2$ as the core curve of a cylinder to complete the proof.  In the remark at the beginning of this section, we permitted saddle connection $6$ to have zero length so that every proof in this section works in the stratum $\cH(2,1,1)$.  The assumption that the translation surface has genus three implies that it is impossible for $\ell(5) = \ell(6) = 0$, so if we permit $\ell(6) \geq 0$, then $\ell(5) > 0$.  We consider a cylinder $C'$ from saddle connection $5$ to itself that lies in the complement of the vertical cylinder from $3$ to itself.  If any such cylinder $C'$ passes through the complement of the horizontal segment in Figure \ref{Config4CylDiagL2Fig2}, then we are done because $b_2$ can be taken to be the core curve of $C'$.  The only obstruction to the existence of such a cylinder $C'$ is if every trajectory from saddle connection $5$ to itself passes through the horizontal segment of length $1/2$.  We claim that Figure \ref{Config4CylDiagL2Fig2} accurately depicts the situation.  By contradiction, if every trajectory from both $3$ to itself and $5$ to itself passes through the horizontal segment, then the total lengths of saddle connections $3$ and $5$, plus whichever saddle connection lies between them (either $6$ or $4$), is at most $1/2$.  In other words, this imposes either the restriction $\ell(3) + \ell(6) + \ell(5) \leq 1/2$ or $\ell(3) + \ell(4) + \ell(5) \leq 1/2$.  Since $\sum_{j=3}^6 \ell(j) = 1$, the relations $\ell(3) + \ell(6) + \ell(5) \leq 1/2$ or $\ell(3) + \ell(4) + \ell(5) \leq 1/2$ imply $\ell(4) \geq 1/2$ or $\ell(6) \geq 1/2$, respectively.  This contradicts the assumption that $3$ is the longest saddle connection, that $\ell(3) \leq 1/2$, and that $\ell(5) > 0$.  Thus, $b_i$ can be realized as the core curve of a cylinder, for all $i$, and we conclude by Corollary \ref{HomBasisForniContraCor}.
\end{proof}

By combining the previous two lemmas, we get

\begin{corollary}
\label{CylDiags4L2}
Let $\mathcal{M}$ be a rank one orbit closure with $2$-dimensional Forni subspace and non-trivial REL.  If $\cM$ contains a translation surface satisfying Cylinder Diagram 4A) Case (L2), then there exists $\cM'$ (in the boundary of $\cM$) in a lower dimensional stratum in genus three with $2$-dimensional Forni subspace.
\end{corollary}

\begin{lemma}
\label{CylDiags4L3}
Let $\mathcal{M}$ be a rank one orbit closure with $2$-dimensional Forni subspace and non-trivial REL.  If $\cM$ contains a translation surface satisfying Cylinder Diagram 4A) Case (L3), then there exists $\cM'$ (in the boundary of $\cM$) in a lower dimensional stratum in genus three with $2$-dimensional Forni subspace.
\end{lemma}

\begin{proof}
We reduce Case (L3) to the Cases (L1) and (L2) above.  Without loss of generality, let $h_1 \leq h_2$.  We let $h_4$ be the free variable and consider the behavior of $h_1$ and $h_3$ as $h_4$ tends to zero.  Since $h_1 + h_3 + h_4 = H_1 > 0$, clearly at least one of $h_1$ and $h_3$ is bounded away from zero as $h_4$ tends to zero.

If neither $h_1$ nor $h_3$ converge to zero as $h_4$ converges to zero, then we can apply the argument of Case (L1).

If $h_3$ converges to zero before $h_4$ converges to zero, then we can switch the roles of $C_4$ and $C_3$ and again apply the argument of Case (L1).

If $h_1$ converges to zero before $h_4$ converges to zero, then Case (L2) applies since at least one of the cylinders $C_1$ and $C_2$ can be made arbitrarily small while the heights of $C_3$ and $C_4$ are bounded away from zero.

Finally, if $h_1$ and $h_4$ simultaneously converge to zero, then Case (L1) can be applied because no assumptions were made about the heights of $C_1$ or $C_2$ in that argument.  This concludes every possibility.
\end{proof}

\begin{remark}
We recall that the complicating issue in this section was that $\cM$ might only be $3$-dimensional.  In the case where $\cM$ is $4$-dimensional, any two elements of $\{h_1, h_3, h_4\}$ can be chosen as the free variables, and therefore, any of the cases above apply.
\end{remark}

Having addressed all possible cases of all possible cylinder configurations for a translation surface in a rank one orbit closure with non-trivial REL and non-trivial Forni subspace, we can finally prove Theorem \ref{Rank1PrinStratum}.

\begin{proof}[Proof of Theorem \ref{Rank1PrinStratum}]
By Lemma \ref{Gen3TopConfigs}, every horizontally periodic translation surface in $\mathcal{M}$ must admit one of the six cylinder configurations listed in Table \ref{MainConfigTable}.  By Lemma \ref{3PlusCyls}, there must be a horizontally periodic translation surface in $\cM$ with at least three cylinders, which reduces the cylinder configurations to Configurations 2), 3) or 4).

Rather than proceeding abstractly by induction on the number of zeros, we count to four.  By Corollary \ref{NoZeroTeichCurvesInH4} and Theorem \ref{NoAISLowerStrata}, there are no orbit closures in $\cH(4)$ with non-trivial Forni subspace.  By contradiction suppose there exists a rank one affine manifold $\cM$ with non-trivial REL.  In strata with two zeros, it is not possible to degenerate to an orbit closure in a lower stratum with non-trivial Forni subspace, so Lemmas \ref{CylDiag2}, \ref{CylDiags3BC}, and \ref{CylDiags3A} prove that for the $3$-cylinder diagrams we must always have $\text{Twist}(M, \mathcal{M}) \not= \text{Pres}(M, \mathcal{M})$.  However, by the remark at the beginning of this section, Configuration 4) is never possible in strata with two zeros, so we have a contradiction that leads us to the case that $\cM$ lies in the stratum with three zeros.

If $\cM \subset \cH(2,1,1)$, then the previous argument proves that for horizontally periodic translation surfaces with exactly three cylinders, $\text{Twist}(M, \mathcal{M}) \not= \text{Pres}(M, \mathcal{M})$.  By Lemma \ref{WrightLem86}, there exists $M' \in \mathcal{M}$ satisfying Configuration 4).  However, if there is a horizontally periodic translation surface satisfying Configuration 4), then it must satisfy Case (L1), (L2), or (L3) by Lemma \ref{CylDiags4B}.  By Lemmas \ref{CylDiags4L1}, \ref{CylDiags4L3}, and Corollary \ref{CylDiags4L2}, none of these cases are possible.  The argument for the principal stratum follows exactly as for the stratum $\cH(2,1,1)$.  Hence, there does not exist a rank one affine manifold with non-trivial REL and non-trivial Forni subspace.
\end{proof}

\section{Rank Two Affine Manifolds in $\mathcal{H}(1,1,1,1)$}
\label{Rank2PrinStrSect}

Theorem \ref{NoAISLowerStrata} proved that there are no rank two affine manifolds with non-trivial Forni subspace outside of the principal stratum.  In this section we prove Theorem \ref{NoRank2PrinStratum}.  The key to the proof is to use the results of \cite{WrightCylDef} for higher rank affine manifolds.

\begin{definition}
Given an affine manifold $\cM$, let $\cC$ be the maximal collection of parallel cylinders on a translation surface $M \in \cM$ such that every cylinder in $\cC$ remains parallel to every other cylinder in $\cC$ under all local deformations of $M \in \cM$.
\end{definition}

\begin{theorem}\cite[Thm. 1.10]{WrightCylDef}
\label{WrightCylDef}
Let $\mathcal{M}$ be an affine invariant submanifold of rank $k$.  Then there is a horizontally periodic translation surface $M \in \mathcal{M}$ whose horizontal core curves span a subspace of $T^M(\mathcal{M})^*$ of dimension $k$.  In particular, there is a horizontally periodic translation surface in $\mathcal{M}$ with at least $k$ horizontal cylinders.
\end{theorem}

The matrices $a_s$ and $u_t$ were defined in Section \ref{RELDefSect}.

\begin{theorem}\cite[Thm. 5.1]{WrightCylDef}
\label{WrightCylDefThm}
Let $M$ be a translation surface, and let $\mathcal{C}$ be an equivalence class of horizontal cylinders on $M$.  Then for all $s,t \in \mathbb{R}$, the surface $a_s^{\mathcal{C}}(u_t^{\mathcal{C}}(\mathcal{M}))$ remains in the $\text{GL}_2(\mathbb{R})$-orbit closure of $M$.
\end{theorem}

\begin{corollary}
\label{Rk2RedConfig}
Let $\cM$ be a rank two affine manifold with non-trivial Forni subspace.  Then $\cM$ contains a translation surface satisfying either Configuration 1), 2), 3), or 4).
\end{corollary}

\begin{proof}
By \cite[Cor. 6]{SmillieWeissMinSets}, every orbit closure contains a horizontally periodic translation surface.  Therefore, it must admit one of the six cylinder configurations of Table \ref{MainConfigTable} by Lemma \ref{Gen3TopConfigs}.  By Theorem \ref{WrightCylDef}, Configurations 5) and 6) can be excluded because there always exists a horizontally periodic translation surface with core curves of cylinders that span a $2$-dimensional subspace of homology in a rank two orbit closure.
\end{proof}

\begin{lemma}
\label{SimpCylImpDistZeros}
If $C$ is a simple cylinder, then it cannot contain the same simple zero in both of its boundaries.
\end{lemma}

\begin{proof}
See Figure \ref{SimpCylDistZerosFig}.
\end{proof}

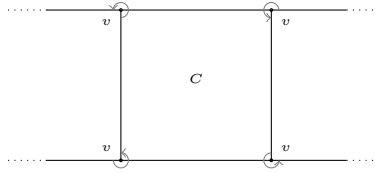
\begin{figure}
\centering
\begin{tikzpicture}[scale=0.5]
\draw (0,0)--(0,4)--(4,4)--(4,0)--cycle;
\draw (-2,0)--(0,0);
\draw [dotted] (-3,0)--(-2,0);
\draw (-2,4)--(0,4);
\draw [dotted] (-3,4)--(-2,4);
\draw (4,0)--(6,0);
\draw [dotted] (6,0)--(7,0);
\draw (4,4)--(6,4);
\draw [dotted] (6,4)--(7,4);

\foreach \x in {(0,0),(0,4),(4,4),(4,0)} \draw \x circle (1pt);

\draw(2,1.8) node[above] {\tiny $C$};

\draw(0,0) node[above left] {\tiny $v$};
\draw(0,4) node[below left] {\tiny $v$};
\draw(4,4) node[below right] {\tiny $v$};
\draw(4,0) node[above right] {\tiny $v$};

\draw [->, gray] (-.2,0) arc [radius=.2, start angle=-180, end angle= 90];
\draw [->, gray] (0,3.8) arc [radius=.2, start angle=-90, end angle= 180];
\draw [->, gray] (4.2,4) arc [radius=.2, start angle=0, end angle= 270];
\draw [->, gray] (4,.2) arc [radius=.2, start angle=90, end angle= 360];

\end{tikzpicture}
 \caption{A zero occurring at all four corners of a simple cylinder $C$ has angle at least $6\pi$.}
 \label{SimpCylDistZerosFig}
\end{figure}

Recall that a cylinder is \emph{free} if it is the only element in its equivalence class of $\mathcal{M}$-parallel cylinders.

\begin{lemma}
\label{H113CylsAllFree}
Let $M$ be a translation surface with three cylinders $C_1, C_2, C_3$ with core curves $\gamma_1, \gamma_2, \gamma_3$, respectively.  If $\gamma_1 + \gamma_2 = \gamma_3$ and the three cylinders split into at least two equivalence classes, then all three cylinders are free.
\end{lemma}

\begin{proof}
If two cylinders are equivalent, then there is a constant $\mu$ relating their core curves.  For example, if $C_1$ and $C_3$ are equivalent, then $\gamma_1 = \mu \gamma_3$.  However, this would imply $\gamma_2 = (1-\mu)\gamma_3$, which implies that all three cylinders are equivalent, a contradiction.  The other combinations are left to the reader.
\end{proof}

We recall a definition from \cite{AulicinoNguyenGen3TwoZeros}.

\begin{definition}
Applying the cylinder stretch $a_s^{\mathcal{C}}$ to the equivalence class $\cC$ on a translation surface and letting $s$ tend to $-\infty$ is called a \emph{cylinder collapse}.
\end{definition}

\begin{lemma}
\label{Rank2Config4}
If $\mathcal{M} \subset \mathcal{H}(1,1,1,1)$ is a rank two affine manifold with non-trivial Forni subspace, then it does not contain a periodic surface admitting the cylinder decomposition described in Configuration 4).
\end{lemma}

\begin{proof}
We claim that every cylinder diagram satisfying Configuration 4) contains a free simple cylinder.  By Lemma \ref{CylDiags4}, it is clear by inspection that every cylinder diagram satisfying Configuration 4) contains a simple cylinder.  By Lemma \ref{Gen3TopConfigs}, this configuration achieves the maximal number of possible cylinders for a translation surface in a rank two affine manifold with non-trivial Forni subspace.  Therefore, the set of cylinders must split into at least two equivalence classes.  Number the cylinders so that $C_3$ and $C_4$ are always homologous, in which case they are equivalent by the definition of cylinder equivalence.  Since the cylinders satisfy $\gamma_{i_1} + \gamma_{i_2} = \gamma_{i_3}$ for an appropriate choice of indices, the proof of Lemma \ref{H113CylsAllFree} implies that every cylinder is free except for those in the equivalence class $\{C_3, C_4\}$.  Hence, the simple cylinder is free.

By Lemma \ref{SimpCylImpDistZeros}, the simple cylinder, say $C_1$, has distinct simple zeros on its boundaries.  Collapse $C_1$ so that the zeros collide.  Since this results in a single saddle connection collapsing that connects distinct zeros, the resulting surface lies in the stratum $\mathcal{H}(2,1,1)$.  By Lemma \ref{DetLocusClosed}, it must also have non-trivial Forni subspace, and by Theorems \ref{NoAISLowerStrata} and \ref{Rank1PrinStratum}, this is not possible and produces the desired contradiction.
\end{proof}

\begin{lemma}
\label{Rank2Config3}
If $\mathcal{M} \subset \mathcal{H}(1,1,1,1)$ is a rank two affine manifold with non-trivial Forni subspace, then it does not contain a periodic surface admitting the cylinder decomposition described in Configuration 3).
\end{lemma}

\begin{proof}
By contradiction, let $M \in \cM$ be a horizontally periodic translation surface satisfying Configuration 3).  There must be at least two equivalence classes of cylinders because $M$ realizes the maximum number of cylinders for a periodic translation surface in $\cM$ by Lemmas \ref{Gen3TopConfigs} and \ref{Rank2Config4}.  Number the cylinders so that $C_2$ and $C_3$ are always homologous, whence equivalent.  Then $C_1$ is a free cylinder.  By inspection, $C_1$ either contains a simple cylinder or it is a simple cylinder.  See Figures \ref{Config3ACylDiagsFig} and \ref{Config3BCCylDiagsFig}.  Furthermore, it is easy to see that if $C_1$ contains a simple cylinder $C$, then there is no cylinder parallel to $C$ that is entirely contained in $C_1$.  By \cite[Prop. 3.3(b)]{NguyenWright}, this simple cylinder is free, and by Lemma \ref{SimpCylImpDistZeros}, it has distinct zeros on top and bottom.

Using Theorem \ref{WrightCylDefThm} to collapse $C$, or $C_1$ if $C_1$ is a simple cylinder, so that the zeros collide, produces a translation surface in $\mathcal{H}(2,1,1)$.  As above, this contradicts Theorems \ref{NoAISLowerStrata} and \ref{Rank1PrinStratum}.
\end{proof}

\begin{lemma}
\label{Rank2Config2}
If $\mathcal{M} \subset \mathcal{H}(1,1,1,1)$ is a rank two affine manifold with non-trivial Forni subspace, then it does not contain a periodic surface admitting the cylinder decomposition described in Configuration 2).
\end{lemma}

\begin{proof}
By contradiction, let $M \in \cM$ be a horizontally periodic translation surface satisfying Configuration 2).  There must be at least two equivalence classes of cylinders because $M$ realizes the maximum number of cylinders for a periodic translation surface in $\cM$ by Lemmas \ref{Gen3TopConfigs} and \ref{Rank2Config4}.  By Lemma \ref{H113CylsAllFree}, all three cylinders are free.  Choose one of the cylinders with a single saddle connection in its base, twist it so that it contains a vertical saddle connection, and collapse it.  The vertical saddle connection was between two distinct zeros, so the resulting translation surface lies in $\mathcal{H}(2,1,1)$.  As above, this contradicts Theorems \ref{NoAISLowerStrata} and \ref{Rank1PrinStratum}.
\end{proof}

\begin{lemma}
\label{Rank2Config1}
If $\mathcal{M} \subset \mathcal{H}(1,1,1,1)$ is a rank two affine manifold with non-trivial Forni subspace, then it does not contain a periodic surface admitting the cylinder decomposition described in Configuration 1).
\end{lemma}

\begin{proof}
By contradiction, let $M \in \cM$ be a horizontally periodic translation surface satisfying Configuration 1).  Since Configuration 1) represents the maximum number of cylinders achievable on a horizontally periodic translation surface in $\mathcal{M}$ by Lemmas \ref{Rank2Config4}, \ref{Rank2Config3}, and \ref{Rank2Config2}, both cylinders must be free.  Since the cylinders are not homologous, one of the cylinders, say $C_1$, must have a saddle connection $\sigma$ on its top and bottom.  Let $C_1' \subset C_1$ be the simple cylinder formed by considering trajectories from $\sigma$ to itself.  Either $C_1'$ is free, or it is not.

If $C'_1$ is not free, then there is a cylinder $C'_2$ parallel to $C'_1$.  By \cite[Prop. 3.2]{NguyenWright}, $C'_2 \subset C_1$, which implies that the closure of $C'_1 \cup C'_2$ is a proper subset of $M$.  By \cite[Cor. 6]{SmillieWeissMinSets}, there exists $M' \in \cM$ such that $C'_1$ and $C'_2$ persist and $M'$ is periodic in the direction of $C'_1$.  Since $C'_1 \cup C'_2$ is a proper subset of $M$, $M'$ must decompose into at least three cylinders in the direction of $C'_1$.  This contradicts Lemmas \ref{Gen3TopConfigs}, \ref{Rank2Config4}, \ref{Rank2Config3}, and \ref{Rank2Config2}, and proves that $C_1'$ is free.

By Lemma \ref{SimpCylImpDistZeros}, $C_1'$ has distinct zeros on top and bottom.  Since $C_1'$ is free with distinct zeros on its top and bottom, we can collapse it to get a translation surface in $\cH(2,1,1)$, and achieve the same contradiction as in every lemma above.  
\end{proof}

\begin{proof}[Proof of Theorem \ref{NoRank2PrinStratum}]
By Corollary \ref{Rk2RedConfig}, every rank two affine manifold with non-trivial Forni subspace in genus three contains a horizontally periodic translation surface satisfying Configuration 1), 2), 3), or 4).  However, these configurations are excluded from existing in such an affine manifold by Lemmas \ref{Rank2Config1}, \ref{Rank2Config2}, \ref{Rank2Config3}, and \ref{Rank2Config4}, respectively.  Thus, there cannot exist a rank two affine manifold in the principal stratum with non-trivial Forni subspace.
\end{proof}

\bibliography{fullbibliotex}{}

\end{document}